\documentclass[10pt, reqno]{amsart}

\allowdisplaybreaks
\usepackage{amssymb,amsthm,amsmath,bm}
\usepackage[numbers,sort&compress]{natbib}

\title[2D Tropical Climate Model]{Initial-boundary value problem for 2D temperature-dependent tropical climate model}
\author{Jitao Liu, \, Yunxiao Zhao}
\address[Jitao Liu]{Department of Mathematics, Faculty of Science, Beijing University of Technology, Beijing, 100124, P. R. China.}
\email{jtliu@bjut.edu.cn,\,\,\,jtliumath@qq.com}
\address[Yunxiao Zhao] {Department of Mathematics, Faculty of Science, Beijing University of Technology, Beijing, 100124, P. R. China.}
\email{yunxiaozhao0826@163.com}


\theoremstyle{plain}
\newtheorem{corollary}{Corollary}[section]
\newtheorem{theorem}{Theorem}[section]
\newtheorem{lemma}{Lemma}[section]
\newtheorem{proposition}{Proposition}[section]

\theoremstyle{definition}
\newtheorem{definition}{Definition}[section]

\newtheorem{remark}{Remark}[section]

\let\f=\frac
\let\p=\partial
\def\R{\mathbb R}

\newcommand{\beq}{\begin{equation}}
\newcommand{\eeq}{\end{equation}}
\newcommand{\ben}{\begin{eqnarray}}
\newcommand{\een}{\end{eqnarray}}
\newcommand{\beno}{\begin{eqnarray*}}
\newcommand{\eeno}{\end{eqnarray*}}

\begin{document}

\begin{abstract}
It is well known that the tropical climate model is an important model to describe the interaction of large scale flow fields and precipitation in the tropical atmosphere. In this paper, we address the issue of global well-posedness for 2D temperature-dependent tropical climate model in a smooth bounded domain. Through classical energy estimates and De Giorgi-Nash-Moser iteration method, we obtain the global existence and uniqueness of strong solution in classical energy spaces. Compared with Cauchy problem, we establish more delicate a priori estimates with exponential decay rates. To the best of our knowledge, this is the first result concerning the global well-posedness for the initial-boundary value problem in 2D tropical climate model.
\end{abstract}
\maketitle

\section{Introduction and main results}
\label{intro}
\setcounter{equation}{0}

In this paper, we consider the following 2D temperature-dependent tropical climate model
\begin{equation}\label{3D}
\left\{\begin{array}{ll}
\bm{u}_t+(\bm{u}\cdot\nabla)\bm{u}-\nabla\cdot(\mu(\theta)\nabla\bm{u})
+\nabla p=-\nabla\cdot(\bm{v}\otimes\bm{v}),\\
\bm{v}_t+(\bm{u}\cdot\nabla)\bm{v}-\nabla\cdot(\nu(\theta)\nabla\bm{v})
=-\nabla\theta-(\bm{v}\cdot\nabla)\bm{u},\\
\theta_t+(\bm{u}\cdot\nabla)\theta-\nabla\cdot(\kappa(\theta)\nabla\theta)=-\nabla\cdot\bm{v},\\
\nabla\cdot\bm{u}=0,
\end{array}\right.
\end{equation}
where $\bm{u}=(u_1(\bm{x},t),u_2(\bm{x},t))$ and $\bm{v}=(v_1(\bm{x},t),v_2(\bm{x},t))$ stand for the barotropic mode and first baroclinic mode of the velocity field respectively, $p=p(\bm{x},t)$ is the scalar pressure and $\theta(\bm{x},t)$ represents the scalar temperature. Here $\bm{v}\otimes\bm{v}=(v_iv_j),\,i,j=1,2$ is the standard tensor notation. In addition, $\mu(\theta)$ and $\nu(\theta)$ represent the viscosities in the momentum equations, $\kappa(\theta)$ denotes the thermal diffusivity, which are assumed to be smooth in $\theta$ and satisfy
\beno
\mu(\theta),\,\nu(\theta),\,\kappa(\theta)\geq\sigma^{-1},\,\,\forall\,\theta\in\mathbb{R},
\eeno
for some constant $\sigma>0.$
\vskip .1in
It is well-known that the barotropic mode is established under the assumption that the isobaric surface and the isothermal surface coincide, which can be understood as an ideal state, while the baroclinic mode exists in the actual atmosphere, which possesses a phenomenon that energy can be converted and the isothermal surface intersects with the isobaric surface. As Majda described in \cite{JA} , momentum transfer between barotropic and baroclinic modes plays a significant role in the study of tropical-extratropical interactions, so it is necessary to retain both barotropic and baroclinic modes of velocity. Detailed physical backgrounds on tropical climate models can be found in \cite{B1,B2,LT0,JA}.
\vskip .1in
In fact, the tropical climate model can be traced back to the model proposed by Friedson-Majda-Pauluis \cite{FMP}, which originates from the inviscid original equations by performing Galerkin truncation to the first baroclinic mode, which also leads to the absence of Laplace term in the equation in \cite{FMP}. If we manage to consider the effect of viscosity on the velocity field and temperature, we can derive the equations \eqref{3D} from the original viscous equations by the same argument as the original inviscid equations. Readers interested in this field can refer to \cite{CC0,CC1,CC2,CC3,CC4,CC5,CC6}.
When $\mu(\theta)$, $\nu(\theta)$ and $\kappa(\theta)$ are positive constants, the global existence of smooth solution to the corresponding system \eqref{3D} can be established using the classical method.
Making use of innovative ideas, the introduction of an pseudo baroclinic velocity and auxiliary function, Li and Titi \cite{LT} established the global well-posedness of strong solution under the condition $\mu(\theta)=\nu(\theta)={\rm const.}>0$ and $\kappa(\theta)=0$. For the same system as \cite{LT}, Li, Zhai and Yin \cite{LZZ} obtained the global well-posedness in Besov spaces with negative order and small initial data. In addition, the tropical climate model with fractional dissipation has also attracted much attention, which can be seen in \cite{D1,D2,D3,D4,Y1,Y2}.
\vskip .1in
However, in some practical applications, such as plasma flow and mantle convection in geophysical fluid dynamics, viscosity and thermal diffusion coefficients are sensitive to temperature changes, usually as a function of $\theta$, so we should pay attention to the dependence on temperature. In this case, due to the existence of nonlinear term, the research of this problem is more challenging. For the whole space, Dong, Li, Xu and Ye \cite{DLXZ} proved that the system \eqref{3D} with general initial data has a unique global smooth solution in $H^s(\mathbb{R}^2), s>1$. Subsequently, Li, Xu and Ye \cite{LXY} studied its asymptotic behavior and obtained the sharp time-decay.
\vskip .1in
It is widely acknowledged that the motion of various fluids in nature usually takes place in bounded regions, and solutions in bounded domains usually exhibit different behaviors and richer phenomena than in the whole space. However, the initial-boundary value problem with respect to \eqref{3D} is still an open problem. In view of mathematics, the method of bounded domains and whole spaces differs when obtaining the desired estimates, mainly in dealing with the boundary terms. Because the spatial derivatives are not sufficient, we have to make full use of the technical means such as Sobolev embedding inequalities and the classical theory of elliptic systems.
\vskip .1in
In this paper, we will make the first attempt to investigate the initial-boundary value problem for the system \eqref{3D}. To this end, we can set the boundary condition of the equations as homogeneous Dirichlet boundary condition, i.e.,
\begin{equation}\label{eq2}
\bm{u}|_{\p \Omega}=0,\quad\bm{v}|_{\p \Omega}=0,\quad\theta|_{\p \Omega}=0
\end{equation}
together with the initial condition
\begin{equation}\label{eq20}
(\bm{u},\bm{v},\theta)(\bm{x},0)=(\bm{u}_0,\bm{v}_0,\theta_0)(\bm{x}),\quad\,\hbox{in}\,\,\Omega,\\
\end{equation}
where $\Omega\subset \R^2$ represents a bounded domain with smooth boundary.
\vskip .1in
Throughout the paper, we regard the letter C as the universal constant, which may vary from line to line. We also write $C(\sigma_1,\sigma_2,\cdots,\sigma_k)$ to denote that C depends on the quantities $\sigma_1,\sigma_2,\cdots,\sigma_k$ and identify $\mathbf{n}=(n_1,n_2)$ as the unit outer normal vector of the boundary $\partial\Omega$.
Moreover, for $1\leq p,r\leq\infty$ and $k\geq1$, the standard Lebesgue and Sobolev spaces can be considered as:
$$L^r=L^r(\Omega),\quad W^{k,r}=W^{k,r}(\Omega),\quad L_t^pL^r=L_t^pL^r(\Omega).$$

Now we are in the position to present the main results of this paper.
\begin{theorem}\label{T1}
Let $\Omega\subset\R^2$ be a bounded domain with smooth boundary.
\begin{enumerate}
\item[\textbf{(a)}] \textbf{Global weak solution:} Suppose $\theta_0\in L^\infty\cap H_0^1$ and $(\bm{u}_0,\bm{v}_0)\in H_0^1$ with $\nabla\cdot\bm{u}_0=0$,
then the system $\eqref{3D}-\eqref{eq20}$ admits a global weak solution.
In addition, for any $t>0$, the solution has the following decay estimate
$$\|\bm{u}(t)\|_{H^1}^2+\|\bm{v}(t)\|_{H^1}^2
+\|\theta(t)\|_{H^1}^2\leq Ce^{-\alpha t},$$
where the constants $C$ and $\alpha$ depend on $C_0,\Omega,\|\theta_0\|_{L^\infty},\|\theta_0\|_{H^1},
\|\bm{u}_0\|_{H^1}$ and $\|\bm{v}_0\|_{H^1}$ only;
\item[\textbf{(b)}] \textbf{Unique global strong solution:} On the basis of \textbf{(a)}, if further $(\bm{u}_0,\bm{v}_0,\theta_0)\in H_0^1\cap H^2$,
then the system $\eqref{3D}-\eqref{eq20}$ has a unique global strong solution $(\bm{u},\bm{v},\theta)$ satisfying
\beno
(\bm{u}_t,\bm{v}_t,\theta_t)\in C(0,\infty;L^2)\cap L^2(0,\infty;H_0^1)
\eeno
and
\beno
(\bm{u},\bm{v},\theta)\in C(0,\infty;H^2)\cap L^2(0,\infty;H^3).
\eeno
Besides, for any $t>0$, we obtain the exponential decay estimates
\beno
\|\bm{u}(t)\|_{H^2}^2+\|\bm{v}(t)\|_{H^2}^2
+\|\theta(t)\|_{H^2}^2\leq Ce^{-\alpha t},
\eeno
where $C$ and $\alpha$ are the constants which depend only on $C_0,\Omega,\|\theta_0\|_{H^2},
\|\bm{u}_0\|_{H^2}$ and $\|\bm{v}_0\|_{H^2}.$
\end{enumerate}
\end{theorem}
\begin{remark}
In contrast to Cauchy problem \cite{DLXZ}, we obtain the estimates independent of time $t$ such that the existence time can be taken to infinity.
\end{remark}
\begin{remark}
 Provided that one imposes more regular assumptions on the initial data and appropriate compatibility conditions, one can show that the solutions are as regular as initial data and decay exponentially in corresponding energy spaces.
\end{remark}
\begin{remark}
If we replace the boundary condition
$\bm{v}|_{\partial\Omega}=0$ with $\nabla\bm{v}\cdot\bm{n}|_{\partial\Omega}=0$,
then the results of Theorem \ref{T1} still hold except for the decay estimates.
\end{remark}
\begin{remark}
When the Dirichlet boundary condition
$\theta|_{\partial\Omega}=0$ is substituted by $\nabla\theta\cdot\bm{n}|_{\partial\Omega}=0$,
then the results of Theorem \ref{T1} still hold except for the decay estimates.
\end{remark}
Now we will make some comments on the analysis of this article. As usual, for the incompressible viscous flows, the key argument to establish the global existence and uniqueness of strong solution is the $L^{\infty}_tH^1$ estimates. However, due to the strong nonlinearity couplings in the right hand sides of system \eqref{3D}, there exists a gap between the basic energy estimates and the $L^{\infty}_tH^1$ estimates. To overcome it, for \eqref{3D}, one possible approach is working on the uniform $L^{\infty}_tL^r$ (with respect to $r$) estimates for $\theta$ through energy estimates. Unfortunately, for any $r\in(2,\infty)$, $\|\theta(t)\|_{L^{\infty}_tL^r}\leq C(r,\bm{u}_0,\bm{v}_0,\theta_0)$, which results in the failure to directly extend $r$ to the infinity.

Based on above analysis and motivated by \cite{DLXZ}, for any $r\in(2,\infty)$, we first build up the $L_t^\infty L^r$ estimates of $\bm{v}$ and then combine with De Giorgi-Nash-Moser iteration method, which allows us to acquire the $L_t^\infty L^\infty$ estimates of $\theta$. Then, in the spirit of \cite{SZ}, by using the good unknown
\ben\label{bigseta}
\hat{\Theta}=\int_0^\theta\kappa(z){\rm d}\bm{z},
\een
we can rewrite the third equation of \eqref{3D} as
\beno
\frac{1}{\kappa(\theta)}(\hat{\Theta}_t+\bm{u}\cdot\nabla\hat{\Theta})-\Delta\hat{\Theta}=-\nabla\cdot\bm{v},
\eeno
which constitutes a stepping stone to the $L_t^\infty H^1$ estimates of $\theta$. In this circumstance, it is accessible for us to attain the $L_t^\infty H^1$ estimates of $(\bm{u},\bm{v})$ and further higher order derivative estimates.

However, quite different from Cauchy problem \cite{DLXZ}, with the purpose of exploring large time behaviors, it is imperative for us to establish more delicate estimates independent of time $t$, which paves the way for exponential decay rate of solutions. In addition, due to the presence of boundary terms and lack of spatial derivatives on the boundary, we will frequently take advantage of the classical elliptic theory and Sobolev embeddings to obtain high order derivative estimates.

Here is the arrangement of the remaining chapters of the article. In section 2, plenty of helpful lemmas and classical results will be revealed. In section 3, we keep a close watch on the global existence of weak solutions. In section 4, we spare no effort to make great achievements in the unique strong solution.

\section{Preliminaries}
\label{prel}
\setcounter{equation}{0}
In this section, we will present some basic facts, including the inequality tools and relevant conclusions of the Stokes system, which will play a significant role in the subsequent sections.
Firstly, we show the definition of the weak solution of system $\eqref{3D}-\eqref{eq20}$.
\begin{definition}\label{d1}
Let $\Omega\subset\R^2$ be a bounded domain with smooth boundary and $(\bm{u},\bm{v},\theta)$ is named a global weak solution of the system
$\eqref{3D}-\eqref{eq20}$, if for any $T>0,\,(\bm{u},\bm{v},\theta)\in C(0,T;L^2)\cap L^2(0,T;H^1)$,\,$(\mu(\theta),\nu(\theta),\\
\kappa(\theta))\in L^2(0,T;L^2)$ and it follows that
\beno
&&\int_\Omega\bm{u}_0\cdot\bm{\varphi}_0{\rm d}\bm{x}+\int_0^T\int_\Omega[\bm{u}\cdot\bm{\varphi}_t
+\bm{u}\cdot\nabla\bm{\varphi}\cdot\bm{u}-\mu(\theta)\nabla\bm{u}:\nabla\bm{\varphi}
+(\bm{v}\otimes\bm{v}):\nabla\bm{\varphi}]{\rm d}\bm{x}{\rm d}t=0,\\
&&\int_\Omega\bm{v}_0\cdot\bm{\psi}_0{\rm d}\bm{x}+\int_0^T\int_\Omega[\bm{v}\cdot\bm{\psi}_t
+\bm{u}\cdot\nabla\bm{\psi}\cdot\bm{v}-\nu(\theta)\nabla\bm{v}:\nabla\bm{\psi}
+(\nabla\cdot\bm{v})\bm{u}\cdot\bm{\psi}+\bm{v}\cdot\nabla\bm{\psi}\cdot\bm{u}]{\rm d}\bm{x}{\rm d}t=0,\\
&&\int_\Omega\theta_0\phi_0{\rm d}\bm{x}+\int_0^T\int_\Omega(\theta\phi_t+\bm{u}\cdot\nabla\phi\theta-\kappa(\theta)
\nabla\theta\cdot\nabla\phi+\bm{v}\cdot\nabla\phi){\rm d}\bm{x}{\rm d}t=0,
\eeno
for any test vector functions $\bm{\varphi}\in C_0^\infty([0,T)\times\Omega)^2$ with $\nabla\cdot\bm{\varphi}=0$, $\bm{\psi}\in C_0^\infty([0,T)\times\Omega)^2$ satisfying $\nabla\cdot\bm{\psi}=0$ and any test scaler functions $\phi\in C_0^\infty([0,T)\times\Omega)$, where $A:B=\sum\limits_{i,j}a_{ij}b_{ij}$.
\end{definition}
The next lemma is the well-known Gagliardo-Nirenberg inequality in bounded domains (see e.g. \cite{GN}).
\begin{lemma}\label{P1}
Let $\Omega\subset\R^n$ be a bounded domain with smooth boundary. Let
$1\leq p, q, r \leq\infty$ be real numbers and $j\le m$ be non-negative
integers. If a real number $\alpha$ satisfies
$$\frac{1}{p} - \frac{j}{n} = \alpha\,\left( \frac{1}{r} - \frac{m}{n} \right) + (1 - \alpha)\frac{1}{q}, \qquad \frac{j}{m} \leq \alpha \leq 1,
$$
then
$$\| \mathrm{D}^{j} f \|_{L^{p}} \leq C_{1} \| \mathrm{D}^{m} f \|_{L^{r}}^{\alpha} \| f \|_{L^{q}}^{1 - \alpha} + C_{2} \|f\|_{L^{s}},$$
where $s > 0$, and the constants $C_1$ and $C_2$ depend upon $\Omega$ and the indices $p,q,r,m,j,s$ only.
\end{lemma}
The following lemma involves an iterative sequence that exerts a huge effect on the proof of Proposition \ref{seta-infinityest} (see \cite{Y3}).
\begin{lemma}\label{giorgi}
Let $a>0$, $b>1$ and $\gamma>1$. Assume that the nonnegative sequence $\{A_k\}_{k=0}^\infty$ satisfies the recurrence relation
\ben\label{yasuo1}
A_{k+1}\leq ab^kA_k^\gamma.
\een
If $A_0$ satisfies
\ben\label{yasuo2}
A_0\leq a^{-\frac{1}{\gamma-1}}b^{-\frac{1}{(\gamma-1)^2}},
\een
then we have $$\lim_{k\to\infty}A_k=0.$$
\end{lemma}

Next we focus on the classical results of regularized estimates for elliptic systems defined on bounded domains, which pave the way for Proposition \ref{uv-H2} (see e.g. \cite{Bian}, \cite{evans}, \cite{L-P-Z-2011}, \cite{SZ}, \cite{Wang}, \cite{Wang2}).
\begin{lemma}\label{wup}
Suppose $\Omega\subset\R^2$ be a bounded domain with smooth boundary and consider the elliptic boundary value problem
\begin{equation}\label{nop}
\left\{\begin{array}{ll}
-\Delta f=g\quad&\hbox{in}\,\,\Omega,\\
f=0\quad&\hbox{on}\,\,\p{\Omega}.
\end{array}\right.
\end{equation}
Then for any $p\in(1,\infty)$ and $g\in W^{m,p}$, \eqref{nop} exists a unique solution $f$ satisfying
\beno
\|f\|_{W^{m+2,p}}\leq C\|g\|_{W^{m,p}},
\eeno
where $m\geq-1$ be an integer and the constant $C=C(\Omega,m,p)$.
\end{lemma}
\begin{remark}\label{tidu}
On the basis of Lemma \ref{wup}, if $m=0$, there holds
$\|\nabla^2 f\|_{L^p}\leq C\|\Delta f\|_{L^p},$ where $C=C(\Omega,p)$.
\end{remark}
Evidently, the 2D Stokes system with variable coefficient in a bounded smooth domain $\Omega\subset\R^2$ is as follows:
\begin{equation}\label{2D}
\left\{\begin{array}{ll}
-\nabla\cdot(\mu(\bm{x})\nabla\bm{u})+\nabla p=\bm{f}\quad&\hbox{in}\,\,\Omega,\\
\nabla\cdot\bm{u}=0\quad&\hbox{in}\,\,\Omega,\\
\bm{u}=0\quad&\hbox{on}\,\,\p{\Omega}.
\end{array}\right.
\end{equation}
Here the coefficient $\mu(\bm{x})$ is a smooth function satisfying
$$0<C_{\min}\leq\mu(\bm{x})\leq C_{\max}<\infty.$$
For any $\bm{f}\in H^{-1}$, there exists a unique weak solution $(\bm{u},p)\in H_0^1\times L^2$ with $\int_\Omega p(\bm{x})d\bm{x}=0$ satisfying
\ben\label{pL2}
\|\bm{u}\|_{H^1}+\|p\|_{L^2}\leq C\|\bm{f}\|_{H^{-1}},
\een
where the constant $C=C(C_{\min},C_{\max},\Omega).$
\vskip .1in
\begin{lemma}\label{nodiv}
Let $(\bm{u},p)\in H^2\times H^1$ be a solution of the Stokes system of non-divergence form
\begin{equation}\label{non-div}
\left\{\begin{array}{ll}
-\mu(\bm{x})\Delta\bm{u}+\nabla p=\bm{f}\quad&\hbox{in}\,\,\Omega,\\
\nabla\cdot\bm{u}=0\quad&\hbox{in}\,\,\Omega,\\
\bm{u}=0\quad&\hbox{on}\,\,\p{\Omega}.
\end{array}\right.
\end{equation}
Then there exists a constant $C=C(C_{\min},C_{\max},\Omega)$ such that
$$\|\nabla^2\bm{u}\|_{L^2}+\|\nabla p\|_{L^2}\leq C(\|\bm{f}\|_{L^2}+\|\bm{u}\|_{H^1}+\|p\|_{L^2}).$$
\end{lemma}

\begin{corollary}\label{C1}
For the solution $(\bm{u},p)$ in Lemma \ref{nodiv}, if one further assumes $\bm{f}\in H^1$, then there holds
$$\|\nabla^3\bm{u}\|_{L^2}+\|\nabla^2p\|_{L^2}\leq C(\|\bm{f}\|_{H^1}+\|\bm{u}\|_{H^2}+\|p\|_{H^1})$$
with $C=C(C_{\min},C_{\max},\Omega).$
\end{corollary}
In order to establish the existence of weak solutions by Schauder fixed point theorem, it is necessary for us to propose the following existence theory as a preparation.
\begin{lemma}\label{fixed}
Suppose $\Omega\subset\R^2$ be a bounded domain with smooth boundary and consider the initial-boundary value problem
\begin{equation}\label{initialbo}
\left\{\begin{array}{ll}
\bm{u}_t+\bm{u}\cdot\nabla\bm{u}-\nabla\cdot(a(x)\nabla\bm{u})+\nabla p=-\nabla\cdot(\bm{v}\otimes\bm{v})\quad&\hbox{in}\,\,\Omega,\\
\bm{v}_t+\bm{u}\cdot\nabla\bm{v}-\nabla\cdot(b(x)\nabla\bm{v})=-\bm{v}\cdot\nabla\bm{u}+{\bm f}\quad&\hbox{in}\,\,\Omega,\\
\nabla\cdot\bm{u}=0\quad&\hbox{in}\,\,\Omega,
\end{array}\right.
\end{equation}
with Dirichlet boundary condition
\begin{equation}\label{initialbo2}
\bm{u}|_{\p \Omega}=0,\quad\bm{v}|_{\p \Omega}=0,
\end{equation}
where the physical meanings of $\bm{u},\bm{v},p$ are the same as those in \eqref{3D}, ${\bm f}$ is an external force term and $a(x), b(x)$ denote variable viscosity coefficients satisfying
$a(x),b(x)\geq\beta>0$. Assume $a(x),b(x)\in C^2(\bar{\Omega})$, the initial data $\bm{u}_0,\bm{v}_0\in C^2(\bar{\Omega})^2$ with $\nabla\cdot\bm{u}_0=0$ and $(\bm{u}_0,\bm{v}_0)|_{\p\Omega}=0$ and ${\bm f}\in C^2([0,T]\times\bar{\Omega})^2$, then there exists a unique smooth solution $(\bm{u},\bm{v},p)$
such that $(\bm{u},\bm{v})\in C^2([0,T]\times\bar{\Omega})^2,\,p\in C^1([0,T]\times\bar{\Omega})$.
\end{lemma}
\begin{remark}
The proof of Lemma \ref{fixed} would be a minor modification to the existence theory of smooth solutions for 2D incompressible MHD system. In fact, \eqref{initialbo} can be thought as its generalized model. To be specific, by assuming $\nabla\cdot\bm{v}=0$, ${\bm f}\equiv0$ and $a(x),b(x)={\rm const.}$ in \eqref{initialbo}, we can view $\bm{v}$ as the magnetic field of 2D incompressible MHD system. So the main differences in the proof potentially lie in non incompressibility of $\bm{v}$, nonconstant of $a(x),b(x)$ and nonzero of ${\bm f}$. However, the a priori estimates established
in Proposition \ref{weak1}, \ref{sita-H2} and \ref{uv-qiangH2} assure us that the differences mentioned above will not destroy the existence theory of smooth solution and one can prove Lemma \ref{fixed} by almost same process as in \cite{JIU}, \cite{KO}. To avoid repetition, we omit it here.
\end{remark}

\section{GLOBAL WEAK SOLUTION}
\label{apriori}
\setcounter{equation}{0}
This section aims to establish the global weak solution. First of all, we establish the a priori estimates, which is crucial and necessary in proving Part \textbf{(a)} of Theorem \ref{T1}.
The main result of this section is stated in the following proposition.
\begin{proposition}\label{weak1}
Let $\Omega\subset\R^2$ be a bounded domain with smooth boundary and
$\theta_0\in L^\infty\cap H_0^1$, $(\bm{u}_0,\bm{v}_0)\in H_0^1$ with $\nabla\cdot\bm{u}_0=0$.
Suppose $(\bm{u},\bm{v},\theta)$ is the solution of system $\eqref{3D}-\eqref{eq20}$, then for any $t>0$, it follows that
\beno
\|\bm{u}(t)\|_{H^1}^2+\|\bm{v}(t)\|_{H^1}^2
+\|\theta(t)\|_{H^1}^2\leq Ce^{-\alpha t},
\eeno
and
\beno
\int_0^te^{\alpha\tau}(\|\bm{u}\|_{H^2}^2+\|\bm{v}\|_{H^2}^2+\|\theta\|_{H^2}^2){\rm d}\tau\leq C,
\eeno
where $C$ and $\alpha$ are the constants relying only on $C_0,\Omega,\|\theta_0\|_{L^\infty},\|\theta_0\|_{H^1},
\|\bm{u}_0\|_{H^1}$ and $\|\bm{v}_0\|_{H^1}.$
\end{proposition}
The proof of Proposition \ref{weak1} can be divided into the following steps.
\subsection{$L^2$ estimates}
\begin{proposition}\label{weak2}
Let $\Omega\subset\R^2$ be a bounded domain with smooth boundary and $(\bm{u},\bm{v},\theta)$ is the solution of system $\eqref{3D}-\eqref{eq20}$. Assuming that $(\bm{u}_0,\bm{v}_0,\theta_0)\in L^2$ , then for any $t>0$, we obtain
\beno
\|\bm{u}(t)\|_{L^2}^2+\|\bm{v}(t)\|_{L^2}^2+\|\theta(t)\|_{L^2}^2
\leq e^{-2\alpha t}(\|\bm{u}_0\|_{L^2}^2+\|\bm{v}_0\|_{L^2}^2+\|\theta_0\|_{L^2}^2)
\eeno
and
\beno
\int_0^te^{\alpha\tau}(\|\nabla\bm{u}\|_{L^2}^2+\|\nabla\bm{v}\|_{L^2}^2
+\|\nabla\theta\|_{L^2}^2){\rm d}\tau
\leq\sigma(\|\bm{u}_0\|_{L^2}^2+\|\bm{v}_0\|_{L^2}^2+\|\theta_0\|_{L^2}^2),
\eeno
with $\alpha=(C^{\ast}\sigma)^{-1}$, where the constant $C^{\ast}$ depends only on the domain $\Omega$.
\end{proposition}
\begin{proof}
Taking the inner product of \eqref{3D} with $(\bm{u},\bm{v},\theta)$ respectively and integrating by parts, we can obtain
\ben\label{uvL2}
&&\f12\f{{\rm d}}{{\rm d}t}(\|\bm{u}(t)\|_{L^2}^2+\|\bm{v}(t)\|_{L^2}^2+\|\theta(t)\|_{L^2}^2)
-\int_{\Omega}\nabla\cdot(\mu(\theta)\nabla\bm{u})\cdot\bm{u}{\rm d}\bm{x}\notag\\
&&-\int_{\Omega}\nabla\cdot(\nu(\theta)\nabla\bm{v})\cdot\bm{v}{\rm d}\bm{x}
-\int_{\Omega}\nabla\cdot(\kappa(\theta)\nabla\theta)\theta{\rm d}\bm{x}
=0.
\een
By direct calculations, it follows that
\beno
&&-\int_\Omega\nabla\cdot(\mu(\theta)\nabla\bm{u})\cdot\bm{u}{\rm d}\bm{x}
=-\int_\Omega\partial_i(\mu(\theta)\partial_iu_j)u_j{\rm d}\bm{x}\\
&=&-\int_{\partial\Omega}\mu(\theta)\partial_iu_ju_jn_i{\rm d}s
+\int_\Omega\mu(\theta)\partial_iu_j\partial_iu_j{\rm d}\bm{x}
\geq\sigma^{-1}\|\nabla\bm{u}\|_{L^2}^2.
\eeno
Identically,
\beno
-\int_\Omega\nabla\cdot(\nu(\theta)\nabla\bm{v})\cdot\bm{v}{\rm d}\bm{x}\geq\sigma^{-1}\|\nabla\bm{v}\|_{L^2}^2,\,
-\int_\Omega\nabla\cdot(\kappa(\theta)\nabla\theta)\theta{\rm d}\bm{x}\geq\sigma^{-1}\|\nabla\theta\|_{L^2}^2.
\eeno
Thus,
\ben\label{decay1}
\f12\f{{\rm d}}{{\rm d}t}(\|\bm{u}(t)\|_{L^2}^2+\|\bm{v}(t)\|_{L^2}^2+\|\theta(t)\|_{L^2}^2)
+\sigma^{-1}(\|\nabla\bm{u}\|_{L^2}^2+\|\nabla\bm{v}\|_{L^2}^2+\|\nabla\theta\|_{L^2}^2)
\leq0.
\een
Applying the Gronwall inequality, it yields that
\ben\label{gronwall1}
\|\bm{u}(t)\|_{L^2}^2+\|\bm{v}(t)\|_{L^2}^2+\|\theta(t)\|_{L^2}^2
+\int_0^t(\|\nabla\bm{u}\|_{L^2}^2+\|\nabla\bm{v}\|_{L^2}^2+
\|\nabla\theta\|_{L^2}^2)(\tau){\rm d}\tau
\leq C(\bm{u}_0,\bm{v}_0,\theta_0).
\een
Owing to the boundary conditions $\bm{u}|_{\partial\Omega}=\bm{v}|_{\partial\Omega}=\theta|_{\partial\Omega}=0,$ we can harness Poincar\'{e} inequality to get
$$\|\bm{u}\|_{L^2}^2\leq C^{\ast}\|\nabla\bm{u}\|_{L^2}^2,\,\,\|\bm{v}\|_{L^2}^2\leq C^{\ast}\|\nabla\bm{v}\|_{L^2}^2,\,\,\|\theta\|_{L^2}^2\leq C^{\ast}\|\nabla\theta\|_{L^2}^2,$$
where the constant $C^{\ast}$ relies on the domain $\Omega$ only.
\vskip .1in
Hence, \eqref{decay1} can be rewritten as
\beno
\f{{\rm d}}{{\rm d}t}(\|\bm{u}(t)\|_{L^2}^2+\|\bm{v}(t)\|_{L^2}^2+\|\theta(t)\|_{L^2}^2)
+\frac{2}{C^{\ast}\sigma}(\|\bm{u}(t)\|_{L^2}^2+\|\bm{v}(t)\|_{L^2}^2+\|\theta(t)\|_{L^2}^2)
\leq0,
\eeno
which derives, after applying the Gronwall inequality, that
\ben\label{gronwall2}
\|\bm{u}(t)\|_{L^2}^2+\|\bm{v}(t)\|_{L^2}^2+\|\theta(t)\|_{L^2}^2
\leq e^{-2\alpha t}(\|\bm{u}_0\|_{L^2}^2+\|\bm{v}_0\|_{L^2}^2+\|\theta_0\|_{L^2}^2),\,\,\forall\,t>0,
\een
with $\alpha=(C^{\ast}\sigma)^{-1}.$
\vskip .1in
Multiplying \eqref{decay1} by $e^{\alpha t}$ and making use of \eqref{gronwall2}, we obtain
\beno
&&\f{{\rm d}}{{\rm d}t}[e^{\alpha t}(\|\bm{u}\|_{L^2}^2+\|\bm{v}\|_{L^2}^2+\|\theta\|_{L^2}^2)]
+\frac{2}{\sigma}e^{\alpha t}(\|\nabla\bm{u}\|_{L^2}^2+\|\nabla\bm{v}\|_{L^2}^2+\|\nabla\theta\|_{L^2}^2)\\
&\leq&\alpha e^{-\alpha t}(\|\bm{u}_0\|_{L^2}^2+\|\bm{v}_0\|_{L^2}^2+\|\theta_0\|_{L^2}^2).
\eeno
After integration over $[0,t],$ we derive
\beno
\int_0^te^{\alpha\tau}(\|\nabla\bm{u}\|_{L^2}^2+\|\nabla\bm{v}\|_{L^2}^2
+\|\nabla\theta\|_{L^2}^2){\rm d}\tau
\leq\sigma(\|\bm{u}_0\|_{L^2}^2+\|\bm{v}_0\|_{L^2}^2+\|\theta_0\|_{L^2}^2),\,\,\forall\,t>0.
\eeno
The proof of Proposition \ref{weak2} is completed here.
\end{proof}

\subsection{$L^r$ estimates}
\begin{proposition}\label{vLr}
Assuming that $(\bm{u}_0,\bm{v}_0,\theta_0)$ satisfies the conditions in Proposition \ref{weak1}, then for $r\in(2,\infty)$, for any $t>0$, it holds that
\ben\label{vLrest}
\|\bm{v}(t)\|_{L^r}+\|\theta(t)\|_{L^r}\leq C(r,\bm{u}_0,\bm{v}_0,\theta_0).
\een
\end{proposition}

\begin{proof}
Multiplying both sides of the second equation in $\eqref{3D}$ by $|\bm{v}|^{r-2}\bm{v}$ and integrating over $\Omega$, we get
\ben\label{vLr-1}
\frac{1}{r}\frac{{\rm d}}{{\rm d}t}\|\bm{v}(t)\|_{L^r}^r-\int_\Omega|\bm{v}|^{r-2}\nabla\cdot(\nu(\theta)\nabla\bm{v})\cdot\bm{v}{\rm d}\bm{x}
=-\int_\Omega|\bm{v}|^{r-2}\bm{v}\cdot(\bm{v}\cdot\nabla)\bm{u}{\rm d}\bm{x}
-\int_\Omega|\bm{v}|^{r-2}\bm{v}\cdot\nabla\theta {\rm d}\bm{x}.
\een
By means of direct calculations, we can derive
\ben\label{vLr-2}
&&-\int_\Omega|\bm{v}|^{r-2}\nabla\cdot(\nu(\theta)\nabla\bm{v})\cdot\bm{v}{\rm d}\bm{x}\notag\\
&=&-\int_{\partial\Omega}|\bm{v}|^{r-2}\nu(\theta)\partial_iv_jv_jn_i{\rm d}s
+\int_\Omega\nu(\theta)\partial_iv_j\partial_i(|\bm{v}|^{r-2}v_j){\rm d}\bm{x}\notag\\
&=&\int_\Omega\nu(\theta)\partial_iv_j\partial_iv_j|\bm{v}|^{r-2}{\rm d}\bm{x}
+(r-2)\int_\Omega\nu(\theta)\partial_iv_j|v|^{r-3}v_j\partial_i|\bm{v}|{\rm d}\bm{x}\notag\\
&=&\int_\Omega\nu(\theta)|\bm{v}|^{r-2}|\nabla\bm{v}|^2{\rm d}\bm{x}
+\frac{4(r-2)}{r^2}\int_\Omega\nu(\theta)(\partial_i|\bm{v}|^{\frac{r}{2}})(\partial_i|\bm{v}|^{\frac{r}{2}}){\rm d}\bm{x}\notag\\
&\geq&\frac{1}{\sigma}\int_\Omega|\bm{v}|^{r-2}|\nabla\bm{v}|^2{\rm d}\bm{x}
+\frac{4(r-2)}{r^2\sigma}\|\nabla|\bm{v}|^{\frac{r}{2}}\|_{L^2}^2.
\een
Taking advantages of H\"{o}lder inequality and Lemma \ref{P1}, it yields that
\ben\label{vLr-3}
&&|-\int_\Omega|\bm{v}|^{r-2}\bm{v}\cdot(\bm{v}\cdot\nabla)\bm{u}{\rm d}\bm{x}|
\leq C\|\nabla\bm{u}\|_{L^2}\||\bm{v}|^{\frac{r}{2}}\|_{L^4}^2\notag\\
&\leq& C\|\nabla\bm{u}\|_{L^2}\||\bm{v}|^{\frac{r}{2}}\|_{L^2}\|\nabla|\bm{v}|^{\frac{r}{2}}\|_{L^2}\notag\\
&\leq&\frac{r-2}{r^2\sigma}\|\nabla|\bm{v}|^{\frac{r}{2}}\|_{L^2}^2+C(r)\|\nabla\bm{u}\|_{L^2}^2\||\bm{v}|^{\frac{r}{2}}\|_{L^2}^2\notag\\
&=&\frac{r-2}{r^2\sigma}\|\nabla|\bm{v}|^{\frac{r}{2}}\|_{L^2}^2+C(r)\|\nabla\bm{u}\|_{L^2}^2\|\bm{v}\|_{L^r}^r.
\een
Applying the Sobolev embedding inequality results in
\ben\label{vLr-4}
&&-\int_\Omega\nabla\theta\cdot(|\bm{v}|^{r-2}\bm{v}){\rm d}\bm{x}
\leq C(r)\int_\Omega|\theta||\bm{v}|^{r-2}|\nabla\bm{v}|{\rm d}\bm{x}\notag\\
&\leq&\frac{1}{2\sigma}\int_\Omega|\bm{v}|^{r-2}|\nabla\bm{v}|^2{\rm d}\bm{x}
+C(r)\|\theta\|_{L^r}^2\|\bm{v}\|_{L^r}^{r-2}\notag\\
&\leq&\frac{1}{2\sigma}\int_\Omega|\bm{v}|^{r-2}|\nabla\bm{v}|^2{\rm d}\bm{x}
+C(r)\|\theta\|_{H^1}^2\|\bm{v}\|_{L^r}^{r-2}.
\een
Notice that $\int_\Omega|\bm{v}|^{r-2}|\nabla\bm{v}|^2{\rm d}\bm{x}$ and $\frac{4(r-2)}{r^2\sigma}\|\nabla|\bm{v}|^{\frac{r}{2}}\|_{L^2}^2$ appearing in \eqref{vLr-2} are non-negative terms, so substituting $\eqref{vLr-2}-\eqref{vLr-4}$ into \eqref{vLr-1}, one can deduce
\beno
\frac{{\rm d}}{{\rm d}t}\|\bm{v}(t)\|_{L^r}^2
\leq C(r)\|\nabla\bm{u}\|_{L^2}^2\|\bm{v}\|_{L^r}^2+C(r)\|\theta\|_{H^1}^2.
\eeno
Making full use of the Sobolev embedding inequality and the Gronwall inequality, it is certain to conclude that
\ben\label{vLr-6}
\|\bm{v}(t)\|_{L^r}^2
\leq(\|\bm{v}_0\|_{L^r}^2+C(r)\int_0^t\|\theta\|_{H^1}^2{\rm d}\tau)e^{\int_0^t\|\nabla\bm{u}\|_{L^2}^2{\rm d}\tau}
\leq C(r,\bm{u}_0,\bm{v}_0,\theta_0).
\een

In terms of the equation of $\theta$, we can use the similar method to deal with it and finally obtain \eqref{vLrest}. To avoid repetition, we omit the elaborate calculations here.
\end{proof}

\subsection{$L^\infty$ estimates of $\theta$}
\
\vskip .1in
It can be seen that we cannot directly obtain the $L_t^\infty L^\infty$ estimates of $\theta$ from \eqref{vLrest}. Fortunately, with the $L_t^\infty L^r$ estimates of $\bm{v}$, we can employ the De Giorgi-Nash-Moser iteration method to obtain the desired estimates, whose core idea is Lemma \ref{giorgi}.
\begin{proposition}\label{seta-infinityest}
Under the presumptions of Proposition \ref{weak1}, then for any $t>0$, it holds that
\ben\label{seta-infinity}
\|\theta\|_{L^\infty}\leq C(\sigma,\bm{u}_0,\bm{v}_0,\theta_0).
\een
\end{proposition}
\begin{proof}
Let $N_k=M(1-2^{-k-1}),k\in N$, $M$ is a positive constant satisfying $M\geq4\|\theta_0\|_{L^\infty}$, which will be given specifically later.
In addition, $\bm{x}_{+}=\max{\{\bm{x},0\}}.$  First, multiplying the third equation of $\eqref{3D}$ by the test function $(\theta-N_k)_{+}$, we achieve
\ben\label{yita-2}
&&\frac{1}{2}\frac{{\rm d}}{{\rm d}t}\|(\theta-N_k)_{+}\|_{L^2}^2
+\frac{1}{\sigma}\int_\Omega|\nabla(\theta-N_k)_{+}|^2{\rm d}\bm{x}\notag\\
&\leq&-\int_\Omega(\nabla\cdot\bm{v})(\theta-N_k)_{+}{\rm d}\bm{x}\notag\\
&=&-\int_{\partial\Omega}v_i(\theta-N_k)_{+}n_i{\rm d}s
+\int_\Omega v_i\partial_i(\theta-N_k)_{+}{\rm d}\bm{x}\notag\\
&\leq&\frac{1}{2\sigma}\|\nabla(\theta-N_k)_{+}\|_{L^2}^2+C\|\bm{v}\cdot1_{\{\theta\geq N_k\}}\|_{L^2}^2,
\een
where we utilize the truth $(\theta-N_k)_{+}|_{\partial\Omega}=0$.
Therefore,
\beno
&&\|(\theta-N_k)_{+}(t)\|_{L_t^\infty L^2}^2
+\frac{1}{\sigma}\|\nabla(\theta-N_k)_{+}\|_{L_t^2L^2}^2
\leq C\int_0^t\|\bm{v}\cdot1_{\{\theta\geq N_k\}}\|_{L^2}^2{\rm d}\tau,
\eeno
where $1_{\{\theta\geq N_k\}}$ represents the characteristic function.
Denoting
\ben\label{yita-30}
A_k=\|(\theta-N_k)_{+}(t)\|_{L_t^\infty L^2}^2+\frac{1}{\sigma}\|\nabla(\theta-N_k)_{+}\|_{L_t^2L^2}^2,
\een
it is clear that
\ben\label{yita-4}
A_k
\leq C\int_0^t\|\bm{v}\cdot1_{\{\theta\geq N_k\}}\|_{L^2}^2{\rm d}\tau
\leq C\|\bm{v}\|_{L_t^6 L^6}^2(\int_0^t|\{\theta\geq N_k\}|{\rm d}\tau)^{\frac{2}{3}}
\leq C\Phi^2(\int_0^t|\{\theta\geq N_k\}|{\rm d}\tau)^{\frac{2}{3}},
\een
where we use the fact
\beno
\|\bm{v}\|_{L_t^6 L^6}\leq C\|\bm{v}\|_{L_t^\infty L^4}^{\frac{2}{3}}\|\nabla\bm{v}\|_{L_t^2 L^2}^{\frac{1}{3}}:=\Phi
\eeno
which can be guaranteed by \eqref{gronwall1}, Lemma \ref{P1} and Proposition \ref{vLr}.

If $\theta\geq N_k$, it shows
\beno
(\theta-N_{k-1})_{+}=(\theta-N_{k})_{+}+N_k-N_{k-1}\geq2^{-k-1}M.
\eeno
Hence,
\ben\label{yita-40}
|\{\theta\geq N_k\}|
\leq\int_{\{\theta\geq N_k\}}\Big[\frac{2^{k+1}}{M}(\theta-N_{k-1})_{+}\Big]^4{\rm d}\bm{x}
\leq\Big(\frac{2^{k+1}}{M}\Big)^4\|(\theta-N_{k-1})_{+}\|_{L^4}^4.
\een
Putting \eqref{yita-40} into \eqref{yita-4} yields that
\ben\label{yita-41}
A_k
&\leq&C\Phi^2\Big(\frac{2^{k+1}}{M}\Big)^{\frac{8}{3}}(\int_0^t\|(\theta-N_{k-1})_{+}\|_{L^4}^4{\rm d}\tau)^{\frac{2}{3}}\notag\\
&=&C\Phi^2\Big(\frac{2^{k+1}}{M}\Big)^{\frac{8}{3}}\|(\theta-N_{k-1})_{+}\|_{L_t^4L^4}^{\frac{8}{3}}\notag\\
&\leq&C\Phi^2\Big(\frac{2^{k+1}}{M}\Big)^{\frac{8}{3}}\Big(\|(\theta-N_{k-1})_{+}\|_{L_t^\infty L^2}^{\frac{1}{2}}\|\nabla(\theta-N_{k-1})_{+}\|_{L_t^2L^2}^{\frac{1}{2}}\Big)^{\frac{8}{3}}\notag\\
&\leq&C_1\Phi^2(\frac{2^{k+1}}{M}\Big)^{\frac{8}{3}}A_{k-1}^{\frac{4}{3}}.
\een

Besides,
\beno
\Big|\Big\{\theta\geq N_0=\frac{M}{2}\Big\}\Big|\leq\Big(\frac{2}{M}\Big)^4\|\theta\|_{L^4}^4,
\eeno
by means of \eqref{yita-4}, one can attain
\ben\label{yita-5}
A_0\leq C\Phi^2\Big(\int_0^t\Big[\Big(\frac{2}{M}\Big)^4\|\theta\|_{L^4}^4\Big]{\rm d}\tau\Big)^{\frac{2}{3}}
\leq C\Phi^2\Big(\frac{2}{M}\Big)^{\frac{8}{3}}\|\theta\|_{L_t^4L^4}^{\frac{8}{3}}.
\een

Applying \eqref{gronwall1} and Lemma \ref{P1}, we can get
\ben\label{yita-6}
\|\theta\|_{L_t^4L^4}\leq C\|\theta\|_{L_t^\infty L^2}^{\frac{1}{2}}\|\nabla\theta\|_{L_t^2L^2}^{\frac{1}{2}}:=\hat{\Phi}.
\een
Thus, we update \eqref{yita-5} as
\ben\label{yita-7}
A_0\leq C_2\Phi^2\Big(\frac{2}{M}\Big)^{\frac{8}{3}}{\hat{\Phi}}^{\frac{8}{3}}.
\een

To apply \eqref{yasuo1} and \eqref{yasuo2}, we take
$a=C_1\Phi^2M^{-\frac{8}{3}}2^\frac{16}{3},\,b=2^{\frac{8}{3}}$ in \eqref{yita-41}
such that
$(C_1\Phi^2M^{-\frac{8}{3}}2^\frac{16}{3})^{-3}(2^{\frac{8}{3}})^{-9}\\
=C_2\Phi^2M^{-\frac{8}{3}}2^\frac{8}{3}{\hat{\Phi}}^{\frac{8}{3}}.$
Therefore, we hold $M=16C_1^{\frac{9}{32}}C_2^{\frac{3}{32}}\Phi^{\frac{3}{4}}\hat{\Phi}^{\frac{1}{4}},$
together with Lemma \ref{giorgi}, which allows us to obtain
$$\lim_{k\rightarrow\infty}A_k=0.$$
In this way, it has
$\theta(\bm{x},t)\leq M.$
\vskip .1in
Applying the same argument to $-\theta$, we also deduce
$\theta(\bm{x},t)\geq-M.$
Thus, we end up with
\ben\label{yita-21}
\|\theta\|_{L^\infty}\leq M.
\een
\end{proof}

In fact, we can draw a conclusion from \eqref{yita-21} that
\ben\label{yita-22}
\mathop{\max}\limits_{\theta\in[-M,M]}
\{|\kappa(\cdot),\kappa'(\cdot),\kappa''(\cdot)|,|\mu(\cdot),\mu'(\cdot),\mu''(\cdot)|,|\sigma(\cdot),\sigma'(\cdot),\sigma''(\cdot)|\}\leq \widetilde{M},
\een
where $\widetilde{M}$ is a constant only depending on $M$.
\vskip .1in
\begin{lemma}\label{setacompare}
With the definition \eqref{bigseta} of $\hat{\Theta}$, for any $q\in[2,\infty]$, it follows that
\beno
\sigma^{-1}\|\nabla\theta\|_{L^q}\leq\|\nabla\hat{\Theta}\|_{L^q}\leq \widetilde{M}\|\nabla\theta\|_{L^q},\\
\sigma^{-1}\|\theta_t\|_{L^q}\leq\|\hat{\Theta}_t\|_{L^q}\leq \widetilde{M}\|\theta_t\|_{L^q}.
\eeno
\end{lemma}
\vskip .1in
\subsection{$H^1$ estimates}
\begin{proposition}\label{seta-L2H2}
Under the presumptions of Proposition \ref{weak1}, then for any $t>0$, it holds that
\beno
\|\nabla\theta\|_{L^2}^2+\|\nabla\hat{\Theta}\|_{L^2}^2\leq Ce^{-\alpha t}
\eeno
and
\beno
\int_0^te^{\alpha\tau}(\|\theta_\tau\|_{L^2}^2+\|\hat{\Theta}_\tau\|_{L^2}^2){\rm d}\tau+
\int_0^te^{\alpha\tau}(\|\nabla^2\theta\|_{L^2}^2+\|\nabla^2\hat{\Theta}\|_{L^2}^2){\rm d}\tau
\leq C,
\eeno
where $C=C(\|\bm{u}_0\|_{L^2},\|\bm{v}_0\|_{L^2},\|\theta_0\|_{H^1},\widetilde{M},\sigma,\alpha).$
\end{proposition}
\begin{proof}
We define $\hat{\Theta}=K(\theta)=\int_0^\theta\kappa(z){\rm d}z,$ then it follows
\ben\label{da-seta}
\hat{\Theta}_t=\kappa(\theta)\theta_t,\quad \partial_i\hat{\Theta}=\kappa(\theta)\partial_i\theta,\quad \Delta\hat{\Theta}=\partial_i(\kappa(\theta)\partial_i\theta).
\een
In this way, we can attain the equation of $\hat{\Theta}$:
\ben\label{da-seta1}
\frac{1}{\kappa(\theta)}(\hat{\Theta}_t+\bm{u}\cdot\nabla\hat{\Theta})-\Delta\hat{\Theta}=-\nabla\cdot\bm{v}.
\een
In addition, the initial condition is
$\hat{\Theta}(\bm{x},0)=K(\theta_0(\bm{x}))\triangleq\hat{\Theta}_0.$
Because of $\theta|_{\p\Omega}=0,$ the boundary condition of the equation $\eqref{da-seta1}$ is
$\hat{\Theta}|_{\p\Omega}=0.$
Taking inner product with $\eqref{da-seta1}$ by $\hat{\Theta}_t$ shows
\ben\label{da-seta3}
\frac{1}{2}\frac{{\rm d}}{{\rm d}t}\|\nabla\hat{\Theta}\|_{L^2}^2+\int_\Omega\frac{1}{\kappa(\theta)}\hat{\Theta}_t^2{\rm d}\bm{x}
=-\int_\Omega\frac{1}{\kappa(\theta)}(\bm{u}\cdot\nabla)\hat{\Theta}\hat{\Theta}_t{\rm d}\bm{x}
-\int_\Omega(\nabla\cdot\bm{v})\hat{\Theta}_t{\rm d}\bm{x}.
\een
By means of Lemma \ref{P1}, we derive
\ben\label{da-seta5}
&&|-\int_\Omega\frac{1}{\kappa(\theta)}(\bm{u}\cdot\nabla)\hat{\Theta}\hat{\Theta}_t{\rm d}\bm{x}|\notag\\
&\leq&\sigma\|\bm{u}\|_{L^4}\|\nabla\hat{\Theta}\|_{L^4}\|\hat{\Theta}_t\|_{L^2}\notag\\
&\leq&C\|\bm{u}\|_{L^2}^{\frac{1}{2}}\|\nabla\bm{u}\|_{L^2}^{\frac{1}{2}}(\|\nabla\hat{\Theta}\|_{L^2}^{\frac{1}{2}}
\|\nabla^2\hat{\Theta}\|_{L^2}^{\frac{1}{2}}+\|\nabla\hat{\Theta}\|_{L^2})\|\hat{\Theta}_t\|_{L^2}.
\een
Putting $\eqref{da-seta5}$ into \eqref{da-seta3}, by virtue of \eqref{yita-22},
we can conclude that
\ben\label{da-seta7}
&&\frac{1}{2}\frac{{\rm d}}{{\rm d}t}\|\nabla\hat{\Theta}\|_{L^2}^2+\frac{\|\hat{\Theta}_t\|_{L^2}^2}{\widetilde{M}}\notag\\
&\leq&C\|\bm{u}\|_{L^2}^{\frac{1}{2}}\|\nabla\bm{u}\|_{L^2}^{\frac{1}{2}}\|\nabla\hat{\Theta}\|_{L^2}^{\frac{1}{2}}
\|\nabla^2\hat{\Theta}\|_{L^2}^{\frac{1}{2}}\|\hat{\Theta}_t\|_{L^2}
+C\|\bm{u}\|_{L^2}^{\frac{1}{2}}\|\nabla\bm{u}\|_{L^2}^{\frac{1}{2}}\|\nabla\hat{\Theta}\|_{L^2}
\|\hat{\Theta}_t\|_{L^2}\notag\\
&&+\,C\|\hat{\Theta}_t\|_{L^2}\|\nabla\bm{v}\|_{L^2}.
\een

Our next goal is to handle $\|\nabla^2\hat{\Theta}\|_{L^2}$, which appears in \eqref{da-seta7}.
On account of the $L^p$ estimates of elliptic equations, together with the equation \eqref{da-seta1}, it yields
\ben\label{da-seta8}
\|\nabla^2\hat{\Theta}\|_{L^2}
&\leq&C\|\hat{\Theta}_t\|_{L^2}+C\|\bm{u}\|_{L^4}\|\nabla\hat{\Theta}\|_{L^4}+\|\nabla\bm{v}\|_{L^2}\notag\\
&\leq&C\|\hat{\Theta}_t\|_{L^2}+C\|\bm{u}\|_{L^2}^{\frac{1}{2}}\|\nabla\bm{u}\|_{L^2}^{\frac{1}{2}}(\|\nabla\hat{\Theta}\|_{L^2}^{\frac{1}{2}}
\|\nabla^2\hat{\Theta}\|_{L^2}^{\frac{1}{2}}+\|\nabla\hat{\Theta}\|_{L^2})+\|\nabla\bm{v}\|_{L^2}\notag\\
&\leq&\frac{1}{2}\|\nabla^2\hat{\Theta}\|_{L^2}+C\|\hat{\Theta}_t\|_{L^2}+C\|\bm{u}\|_{L^2}\|\nabla\bm{u}\|_{L^2}\|\nabla\hat{\Theta}\|_{L^2}+\|\nabla\bm{v}\|_{L^2}.
\een
Plugging \eqref{da-seta8} into \eqref{da-seta7}, one can show
\ben\label{da-seta9}
\frac{{\rm d}}{{\rm d}t}\|\nabla\hat{\Theta}\|_{L^2}^2+\frac{\|\hat{\Theta}_t\|_{L^2}^2}{\widetilde{M}}
\leq C\|\bm{u}\|_{L^2}^2\|\nabla\bm{u}\|_{L^2}^2\|\nabla\hat{\Theta}\|_{L^2}^2+C\|\nabla\bm{v}\|_{L^2}^2.
\een

Multiplying \eqref{da-seta9} with $e^{\alpha t}$, we derive
\ben\label{da-seta113}
&&\frac{{\rm d}}{{\rm d}t}(e^{\alpha t}\|\nabla\hat{\Theta}\|_{L^2}^2)+\frac{1}{\widetilde{M}}e^{\alpha t}\|\hat{\Theta}_t\|_{L^2}^2\notag\\
&\leq&C e^{\alpha t}\|\bm{u}\|_{L^2}^2\|\nabla\bm{u}\|_{L^2}^2\|\nabla\hat{\Theta}\|_{L^2}^2+C e^{\alpha t}\|\nabla\bm{v}\|_{L^2}^2+C e^{\alpha t}\|\nabla\hat{\Theta}\|_{L^2}^2\notag\\
&\leq&C e^{\alpha t}\|\bm{u}\|_{L^2}^2\|\nabla\bm{u}\|_{L^2}^2\|\nabla\hat{\Theta}\|_{L^2}^2+C e^{\alpha t}\|\nabla\bm{v}\|_{L^2}^2+C e^{\alpha t}\|\nabla\theta\|_{L^2}^2.
\een
Adopting the Gronwall inequality, in line with \eqref{gronwall1} and Lemma \ref{setacompare}, we net
\ben\label{da-seta1113}
e^{\alpha t}(\|\nabla\hat{\Theta}\|_{L^2}^2+\|\nabla\theta\|_{L^2}^2)+\frac{1}{\widetilde{M}}\int_0^te^{\alpha \tau}(\|\hat{\Theta}_\tau\|_{L^2}^2+\|\theta_\tau\|_{L^2}^2){\rm d}\tau
\leq C.
\een

Multiplying \eqref{da-seta8} with $e^{\alpha t}$ and integrating in time over $[0,t]$ lead to
\ben\label{da-seta17}
\int_0^te^{\alpha\tau}\|\nabla^2\hat{\Theta}\|_{L^2}^2{\rm d}\tau
&\leq&C\int_0^te^{\alpha\tau}\|\hat{\Theta}_\tau\|_{L^2}^2{\rm d}\tau+C\int_0^te^{-2\alpha\tau}\|\nabla\bm{u}\|_{L^2}^2{\rm d}\tau
+C\int_0^te^{\alpha\tau}\|\nabla\bm{v}\|_{L^2}^2{\rm d}\tau\notag\\
&\leq&C,
\een
here \eqref{gronwall2} and \eqref{da-seta1113} exert an influence on it.

By virtue of \eqref{da-seta}, we know
\beno
\partial_i\partial_j\theta
=\frac{\kappa(\theta)\partial_i\partial_j\hat{\Theta}-\kappa'(\theta)\partial_j\theta\partial_i\hat{\Theta}}{[\kappa(\theta)]^2}
=\frac{1}{\kappa(\theta)}\partial_i\partial_j\hat{\Theta}-\frac{1}{[\kappa(\theta)]^3}\kappa'(\theta)\partial_j\hat{\Theta}\partial_i\hat{\Theta},
\eeno
which indicates that
\ben\label{H2-4}
\|\partial_i\partial_j\theta\|_{L^2}
\leq\sigma\|\partial_i\partial_j\hat{\Theta}\|_{L^2}+\sigma^3\widetilde{M}\|\partial_i\hat{\Theta}\|_{L^4}\|\partial_j\hat{\Theta}\|_{L^4}
\leq C(\|\nabla^2\hat{\Theta}\|_{L^2}+\|\nabla\hat{\Theta}\|_{L^2}\|\nabla^2\hat{\Theta}\|_{L^2}+\|\nabla\hat{\Theta}\|_{L^2}^2).
\een
The equation \eqref{H2-4} is summed over $i,j$ and squared to give
\ben\label{H2-5}
\|\nabla^2\theta\|_{L^2}^2
&\leq&C(\|\nabla^2\hat{\Theta}\|_{L^2}^2+\|\nabla\hat{\Theta}\|_{L^2}^2\|\nabla^2\hat{\Theta}\|_{L^2}^2+\|\nabla\hat{\Theta}\|_{L^2}^4),
\een
which implies
\beno
\int_0^te^{\alpha\tau}\|\nabla^2\theta\|_{L^2}^2{\rm d}\tau
\leq C.
\eeno
From what has been discussed above, we draw a conclusion that
\beno
\int_0^te^{\alpha\tau}(\|\theta_\tau\|_{L^2}^2+\|\hat{\Theta}_\tau\|_{L^2}^2){\rm d}\tau+
\int_0^te^{\alpha\tau}(\|\nabla^2\theta\|_{L^2}^2+\|\nabla^2\hat{\Theta}\|_{L^2}^2){\rm d}\tau
\leq C,
\eeno
where $C=C(\|\bm{u}_0\|_{L^2},\|\bm{v}_0\|_{L^2},\|\theta_0\|_{L^2},\|\nabla\theta_0\|_{L^2},\widetilde{M},\sigma,\alpha).$
This completes the proof of Proposition \ref{seta-L2H2}.
\end{proof}

\begin{proposition}\label{uv-H2}
Under the presumptions of Proposition \ref{weak1}, then for any $t>0$, it holds
\beno
\|\nabla\bm{u}\|_{L^2}^2+\|\nabla\bm{v}\|_{L^2}^2\leq Ce^{-\alpha t}
\eeno
and
\beno
\int_0^te^{\alpha\tau}(\|\bm{u}_\tau\|_{L^2}^2+\|\bm{v}_\tau\|_{L^2}^2){\rm d}\tau+\int_0^te^{\alpha\tau}(\|\nabla^2\bm{u}\|_{L^2}^2+\|\nabla^2\bm{v}\|_{L^2}^2){\rm d}\tau
\leq C,
\eeno
where $C=C(\|\bm{u}_0\|_{H^1},\|\bm{v}_0\|_{H^1},\|\theta_0\|_{H^1},\widetilde{M},\sigma,\alpha).$
\end{proposition}
\begin{proof}
Taking inner product with the first two equations of $\eqref{3D}$ by $(\bm{u}_t,\bm{v}_t)$ respectively and integrating by parts, we get
\ben\label{uvH21}
&&\frac{1}{2\sigma}\frac{{\rm d}}{{\rm d}t}(\|\nabla\bm{u}\|_{L^2}^2+\|\nabla\bm{v}\|_{L^2}^2)+\|\bm{u}_t\|_{L^2}^2+\|\bm{v}_t\|_{L^2}^2\notag\\
&\leq&-\int_\Omega\bm{u}_t\cdot(\bm{u}\cdot\nabla)\bm{u}{\rm d}\bm{x}
-\int_\Omega\nabla\cdot(\bm{v}\otimes\bm{v})\cdot\bm{u}_t{\rm d}\bm{x}
-\int_\Omega(\bm{u}\cdot\nabla)\bm{v}\cdot\bm{v}_t{\rm d}\bm{x}\notag\\
&&-\int_\Omega\nabla\theta\cdot\bm{v}_t{\rm d}\bm{x}
-\int_\Omega(\bm{v}\cdot\nabla)\bm{u}\cdot\bm{v}_t{\rm d}\bm{x}.
\een
In line with \eqref{gronwall1} and Lemma \ref{P1}, we obtain
\ben\label{uvH24}
&&|-\int_\Omega\bm{u}_t\cdot(\bm{u}\cdot\nabla)\bm{u}{\rm d}\bm{x}|
\leq C\|\bm{u}_t\|_{L^2}\|\bm{u}\|_{L^4}\|\nabla\bm{u}\|_{L^4}\notag\\
&\leq&C\|\bm{u}_t\|_{L^2}\|\bm{u}\|_{L^2}^{\frac{1}{2}}\|\nabla\bm{u}\|_{L^2}^{\frac{1}{2}}(\|\nabla\bm{u}\|_{L^2}^{\frac{1}{2}}\|\nabla^2\bm{u}\|_{L^2}^{\frac{1}{2}}+\|\nabla\bm{u}\|_{L^2})\notag\\
&\leq&C\|\bm{u}_t\|_{L^2}\|\nabla\bm{u}\|_{L^2}\|\nabla^2\bm{u}\|_{L^2}^{\frac{1}{2}}+C\|\bm{u}_t\|_{L^2}\|\nabla\bm{u}\|_{L^2}^{\frac{3}{2}}.
\een
Likewise, it has
\ben\label{uvH25}
&&|-\int_\Omega\nabla\cdot(\bm{v}\otimes\bm{v})\cdot\bm{u}_t{\rm d}\bm{x}|
\leq C\|\bm{u}_t\|_{L^2}\|\bm{v}\|_{L^4}\|\nabla\bm{v}\|_{L^4}\notag\\
&\leq&C\|\bm{u}_t\|_{L^2}\|\bm{v}\|_{L^2}^{\frac{1}{2}}\|\nabla\bm{v}\|_{L^2}^{\frac{1}{2}}(\|\nabla\bm{v}\|_{L^2}^{\frac{1}{2}}\|\nabla^2\bm{v}\|_{L^2}^{\frac{1}{2}}+\|\nabla\bm{v}\|_{L^2})\notag\\
&\leq&C\|\bm{u}_t\|_{L^2}\|\nabla\bm{v}\|_{L^2}\|\nabla^2\bm{v}\|_{L^2}^{\frac{1}{2}}+C\|\bm{u}_t\|_{L^2}\|\nabla\bm{v}\|_{L^2}^{\frac{3}{2}}.
\een
Putting $\eqref{uvH24},\,\eqref{uvH25}$ into \eqref{uvH21}, we can deduce
\beno
&&\frac{1}{2\sigma}\frac{{\rm d}}{{\rm d}t}(\|\nabla\bm{u}\|_{L^2}^2+\|\nabla\bm{v}\|_{L^2}^2)+\|\bm{u}_t\|_{L^2}^2+\|\bm{v}_t\|_{L^2}^2\\
&\leq&C\|\bm{u}_t\|_{L^2}\|\nabla\bm{u}\|_{L^2}\|\nabla^2\bm{u}\|_{L^2}^{\frac{1}{2}}+C\|\bm{u}_t\|_{L^2}\|\nabla\bm{u}\|_{L^2}^{\frac{3}{2}}
+C\|\bm{u}_t\|_{L^2}\|\nabla\bm{v}\|_{L^2}\|\nabla^2\bm{v}\|_{L^2}^{\frac{1}{2}}+C\|\bm{u}_t\|_{L^2}\|\nabla\bm{v}\|_{L^2}^{\frac{3}{2}}\\
&&+C\|\bm{v}_t\|_{L^2}\|\nabla\bm{u}\|_{L^2}^{\frac{1}{2}}\|\nabla\bm{v}\|_{L^2}^{\frac{1}{2}}\|\nabla^2\bm{v}\|_{L^2}^{\frac{1}{2}}
+C\|\bm{v}_t\|_{L^2}\|\nabla\bm{u}\|_{L^2}^{\frac{1}{2}}\|\nabla\bm{v}\|_{L^2}+C\|\bm{v}_t\|_{L^2}\|\nabla\bm{v}\|_{L^2}^{\frac{1}{2}}\|\nabla\bm{u}\|_{L^2}^{\frac{1}{2}}\|\nabla^2\bm{u}\|_{L^2}^{\frac{1}{2}}\\
&&+C\|\bm{v}_t\|_{L^2}\|\nabla\bm{v}\|_{L^2}^{\frac{1}{2}}\|\nabla\bm{u}\|_{L^2}+C\|\nabla\theta\|_{L^2}\|\bm{v}_t\|_{L^2}\\
&\leq&\frac{1}{2}(\|\bm{u}_t\|_{L^2}^2+\|\bm{v}_t\|_{L^2}^2)
+C(\|\nabla\bm{u}\|_{L^2}^2+\|\nabla\bm{v}\|_{L^2}^2)
(\|\nabla^2\bm{u}\|_{L^2}+\|\nabla^2\bm{v}\|_{L^2}
+\|\nabla\theta\|_{L^2}^2+\|\nabla\bm{u}\|_{L^2}+\|\nabla\bm{v}\|_{L^2}),
\eeno
which suggests
\ben\label{uvH27}
&&\frac{1}{\sigma}\frac{{\rm d}}{{\rm d}t}(\|\nabla\bm{u}\|_{L^2}^2+\|\nabla\bm{v}\|_{L^2}^2)+\|\bm{u}_t\|_{L^2}^2+\|\bm{v}_t\|_{L^2}^2\notag\\
&\leq&C(\|\nabla\bm{u}\|_{L^2}^2+\|\nabla\bm{v}\|_{L^2}^2)
(\|\nabla^2\bm{u}\|_{L^2}+\|\nabla^2\bm{v}\|_{L^2}+\|\nabla\theta\|_{L^2}^2
+\|\nabla\bm{u}\|_{L^2}+\|\nabla\bm{v}\|_{L^2}).
\een

The next step is to deal with $\|\nabla^2\bm{u}\|_{L^2}$ and $\|\nabla^2\bm{v}\|_{L^2}$.
According to Lemma \ref{nodiv}, \eqref{pL2} and Poincar\'{e} inequality, for $\eqref{3D}_1$, we have
\ben\label{uvH28}
&&\|\nabla^2\bm{u}\|_{L^2}+\|\nabla p\|_{L^2}\notag\\
&\leq&C(\|\bm{u}_t\|_{L^2}+\|\bm{u}\cdot\nabla\bm{u}\|_{L^2}+\|\nabla\cdot(\bm{v}\otimes\bm{v})\|_{L^2}
+\|\nabla\theta\cdot\nabla\bm{u}\|_{L^2}+\|\bm{u}\|_{H^1})\notag\\
&\leq&C(\|\bm{u}_t\|_{L^2}+\|\bm{u}\|_{L^4}\|\nabla\bm{u}\|_{L^4}
+\|\bm{v}\|_{L^4}\|\nabla\bm{v}\|_{L^4}
+\|\nabla\theta\|_{L^4}\|\nabla\bm{u}\|_{L^4}+\|\nabla\bm{u}\|_{L^2})\notag\\
&\leq&C\Big[\|\bm{u}_t\|_{L^2}+\|\bm{u}\|_{L^2}^{\frac{1}{2}}\|\nabla\bm{u}\|_{L^2}^{\frac{1}{2}}(\|\nabla\bm{u}\|_{L^2}^{\frac{1}{2}}\|\nabla^2\bm{u}\|_{L^2}^{\frac{1}{2}}+\|\nabla\bm{u}\|_{L^2})
+\|\bm{v}\|_{L^2}^{\frac{1}{2}}\|\nabla\bm{v}\|_{L^2}^{\frac{1}{2}}(\|\nabla\bm{v}\|_{L^2}^{\frac{1}{2}}\|\nabla^2\bm{v}\|_{L^2}^{\frac{1}{2}}+\|\nabla\bm{v}\|_{L^2})\notag\\
&&+(\|\nabla\theta\|_{L^2}^{\frac{1}{2}}\|\nabla^2\theta\|_{L^2}^{\frac{1}{2}}+\|\nabla\theta\|_{L^2})(\|\nabla\bm{u}\|_{L^2}^{\frac{1}{2}}\|\nabla^2\bm{u}\|_{L^2}^{\frac{1}{2}}+\|\nabla\bm{u}\|_{L^2})
+\|\nabla\bm{u}\|_{L^2}\Big]\notag\\
&\leq&\frac{1}{2}\|\nabla^2\bm{u}\|_{L^2}+\frac{1}{4}\|\nabla^2\bm{v}\|_{L^2}
+C(\|\bm{u}_t\|_{L^2}+\|\nabla\bm{u}\|_{L^2}^2+\|\nabla\bm{v}\|_{L^2}^2+\|\nabla\theta\|_{L^2}\|\nabla^2\theta\|_{L^2}\|\nabla\bm{u}\|_{L^2}\notag\\
&&+\|\nabla\theta\|_{L^2}^2\|\nabla\bm{u}\|_{L^2}),
\een
which implies that
\ben\label{uvH29}
\|\nabla^2\bm{u}\|_{L^2}
\leq\frac{1}{2}\|\nabla^2\bm{v}\|_{L^2}
+C(\|\bm{u}_t\|_{L^2}+\|\nabla\bm{u}\|_{L^2}^2+\|\nabla\bm{v}\|_{L^2}^2
+\|\nabla\theta\|_{L^2}\|\nabla^2\theta\|_{L^2}\|\nabla\bm{u}\|_{L^2}
+\|\nabla\theta\|_{L^2}^2\|\nabla\bm{u}\|_{L^2}).
\een
Write the equation of $\bm{v}$ as
\ben\label{uvH210}
-\nu(\theta)\Delta\bm{v}=\nu'(\theta)\nabla\theta\cdot\nabla\bm{v}-\bm{v}_t-(\bm{u}\cdot\nabla)\bm{v}-\nabla\theta-(\bm{v}\cdot\nabla)\bm{u}.
\een
Recalling Lemma \ref{wup} and \eqref{yita-22}, we arrive at
\beno
&&\|\nabla^2\bm{v}\|_{L^2}\\
&\leq&C(\|\nabla\theta\cdot\nabla\bm{v}\|_{L^2}+\|\bm{v}_t\|_{L^2}+\|\bm{u}\cdot\nabla\bm{v}\|_{L^2}+\|\nabla\theta\|_{L^2}+\|\bm{v}\cdot\nabla\bm{u}\|_{L^2})\\
&\leq&C(\|\nabla\theta\|_{L^4}\|\nabla\bm{v}\|_{L^4}+\|\bm{v}_t\|_{L^2}+\|\bm{u}\|_{L^4}\|\nabla\bm{v}\|_{L^4}+\|\nabla\theta\|_{L^2}+\|\bm{v}\|_{L^4}\|\nabla\bm{u}\|_{L^4})\\
&\leq&C(\|\nabla\theta\|_{L^2}^{\frac{1}{2}}\|\nabla^2\theta\|_{L^2}^{\frac{1}{2}}+\|\nabla\theta\|_{L^2})(\|\nabla\bm{v}\|_{L^2}^{\frac{1}{2}}\|\nabla^2\bm{v}\|_{L^2}^{\frac{1}{2}}+\|\nabla\bm{v}\|_{L^2})
+C\|\bm{v}_t\|_{L^2}\\
&&+C\|\bm{u}\|_{L^2}^{\frac{1}{2}}\|\nabla\bm{u}\|_{L^2}^{\frac{1}{2}}(\|\nabla\bm{v}\|_{L^2}^{\frac{1}{2}}\|\nabla^2\bm{v}\|_{L^2}^{\frac{1}{2}}+\|\nabla\bm{v}\|_{L^2})+C\|\nabla\theta\|_{L^2}
+C\|\bm{v}\|_{L^2}^{\frac{1}{2}}\|\nabla\bm{v}\|_{L^2}^{\frac{1}{2}}\\
&&\times(\|\nabla\bm{u}\|_{L^2}^{\frac{1}{2}}\|\nabla^2\bm{u}\|_{L^2}^{\frac{1}{2}}+\|\nabla\bm{u}\|_{L^2})\\
&\leq&\frac{1}{2}\|\nabla^2\bm{v}\|_{L^2}+\frac{1}{2}\|\nabla^2\bm{u}\|_{L^2}+C(\|\bm{v}_t\|_{L^2}+\|\nabla\bm{u}\|_{L^2}^2+\|\nabla\bm{v}\|_{L^2}^2+\|\nabla\theta\|_{L^2}\|\nabla^2\theta\|_{L^2}\|\nabla\bm{v}\|_{L^2}\\
&&+\|\nabla\theta\|_{L^2}^2\|\nabla\bm{v}\|_{L^2}+\|\nabla\theta\|_{L^2}),
\eeno
which suggests that
\beno
\|\nabla^2\bm{v}\|_{L^2}
\leq\|\nabla^2\bm{u}\|_{L^2}+C(\|\bm{v}_t\|_{L^2}+\|\nabla\bm{u}\|_{L^2}^2+\|\nabla\bm{v}\|_{L^2}^2
+\|\nabla\theta\|_{L^2}\|\nabla^2\theta\|_{L^2}\|\nabla\bm{v}\|_{L^2}+\|\nabla\theta\|_{L^2}^2\|\nabla\bm{v}\|_{L^2}+\|\nabla\theta\|_{L^2}).
\eeno
Recalling \eqref{uvH29}, we get
\begin{align}\label{uvH211}
\|\nabla^2\bm{v}\|_{L^2}&\leq C(\|\bm{u}_t\|_{L^2}+\|\bm{v}_t\|_{L^2}+\|\nabla\bm{u}\|_{L^2}^2+\|\nabla\bm{v}\|_{L^2}^2
+\|\nabla\theta\|_{L^2}\|\nabla^2\theta\|_{L^2}\|\nabla\bm{u}\|_{L^2}+\|\nabla\theta\|_{L^2}\|\nabla^2\theta\|_{L^2}\|\nabla\bm{v}\|_{L^2}\notag\\
&+\|\nabla\theta\|_{L^2}^2\|\nabla\bm{u}\|_{L^2}+\|\nabla\theta\|_{L^2}^2\|\nabla\bm{v}\|_{L^2}
+\|\nabla\theta\|_{L^2}).
\end{align}
Putting \eqref{uvH29}, \eqref{uvH211} into \eqref{uvH27}, it holds
\ben\label{houjiade}
&&\frac{{\rm d}}{{\rm d}t}(\|\nabla\bm{u}\|_{L^2}^2+\|\nabla\bm{v}\|_{L^2}^2)+\frac{\sigma}{2}(\|\bm{u}_t\|_{L^2}^2+\|\bm{v}_t\|_{L^2}^2)\notag\\
&\leq& C(\|\nabla\bm{u}\|_{L^2}^2+\|\nabla\bm{v}\|_{L^2}^2)(\|\nabla\bm{u}\|_{L^2}^2+\|\nabla\bm{v}\|_{L^2}^2
+\|\nabla\theta\|_{L^2}^2\|\nabla^2\theta\|_{L^2}^2+\|\nabla\theta\|_{L^2}^4).
\een

Multiplying \eqref{houjiade} by $e^{\alpha t}$ and taking advantage of Gronwall inequality, we deduce
\ben\label{houjiade3}
e^{\alpha t}(\|\nabla\bm{u}\|_{L^2}^2+\|\nabla\bm{v}\|_{L^2}^2)+\frac{\sigma}{2}
\int_0^te^{\alpha\tau}(\|\bm{u}_\tau\|_{L^2}^2+\|\bm{v}_\tau\|_{L^2}^2){\rm d}\tau
\leq C,
\een
where $C=C(\|\bm{u}_0\|_{H^1},\|\bm{v}_0\|_{H^1},\|\theta_0\|_{H^1},\widetilde{M},\sigma,\alpha).$

Recalling \eqref{uvH29} and \eqref{uvH211}, it is easy to get
\beno
\int_0^te^{\alpha\tau}(\|\nabla^2\bm{u}\|_{L^2}^2+\|\nabla^2\bm{v}\|_{L^2}^2){\rm d}\tau\leq C.
\eeno
\end{proof}

\begin{proof}[\bf Proof of Part \textbf{(a)} of Theorem \ref{T1}:]
The proof is a consequence of Schauder's fixed point theorem.

To define the functional setting, we fix $T>0$ and $R_0$, which can
be specified later. For notational convenience, we
write
$$
X\equiv C(0,T; L^2)\cap L^2(0,T; H_0^1)
$$
with $\|f\|_X\equiv \|f\|_{C(0,T; L^2)}+\|f\|_{L^2(0,T; H_0^1)}$
and define
$$
D=\{f\in X\,|\,\|f\|_X\leq R_0\}.
$$
Obviously, $D\subset X$ is closed and convex.

We fix $\epsilon\in(0, 1)$ and define a continuous map on $D$. For any $f\in D$, we regularize it and the initial data $(\bm{u}_0, \bm{v}_0, \theta_0)$ via the
standard mollifying process,
$$
f^{\epsilon}= \rho^{\epsilon}\ast f, \quad \bm{u}_{0}^{\epsilon} = \rho^{\epsilon}\ast \bm{u}_0, \quad \bm{v}_{0}^{\epsilon} = \rho^{\epsilon}\ast \bm{v}_0,
\quad \theta_{0}^{\epsilon} = \rho^{\epsilon}\ast \theta_0,
$$
where $\rho^{\epsilon}$ is the standard mollifier. Then it holds that
$$f^{\epsilon}\in C(0,T;C_0^\infty(\Omega)), \|f^{\epsilon}\|_X\leq\|f\|_X,$$
$$\bm{u}_{0}^{\epsilon}\in C_0^\infty(\Omega), \nabla\cdot\bm{u}_{0}^{\epsilon}=0, \|\bm{u}_{0}^{\epsilon}-\bm{u}_{0}\|_{H^1}<\epsilon,$$
$$\bm{v}_{0}^{\epsilon}\in C_0^\infty(\Omega), \|\bm{v}_{0}^{\epsilon}-\bm{v}_{0}\|_{H^1}<\epsilon,$$
$$\theta_0^{\epsilon}\in C_0^\infty(\Omega),\|\theta_0^{\epsilon}-\theta_0\|_{H^1}<\epsilon.$$

According to Lemma \ref{fixed}, the following system with smooth external forcing $\nabla f^{\epsilon}$ and smooth initial data $(\bm{u}_{0}^{\epsilon}, \bm{v}_0^{\epsilon})$
\begin{equation}\label{molify}
\left\{\begin{array}{ll}
\bm{u}_t+(\bm{u}\cdot\nabla)\bm{u}-\nabla\cdot(\mu(f^{\epsilon})\nabla\bm{u})+\nabla p=-\nabla\cdot(\bm{v}\otimes\bm{v}),\\
\bm{v}_t+(\bm{u}\cdot\nabla)\bm{v}-\nabla\cdot(\nu(f^{\epsilon})\nabla\bm{v})=-\nabla f^{\epsilon}-(\bm{v}\cdot\nabla)\bm{u},\\
\nabla\cdot\bm{u}=0,\,\bm{u}|_{\p\Omega}=\bm{v}|_{\p\Omega}=0,\\
\bm{u}(\bm{x},0)=\bm{u}_{0}^{\epsilon}, \bm{v}(\bm{x},0)=\bm{v}_{0}^{\epsilon},
\end{array}\right.
\end{equation}
has a unique smooth solution $\bm{u}^{\epsilon}, \bm{v}^{\epsilon}$.
At present, $\bm{u}^{\epsilon}$ and $\bm{v}^{\epsilon}$ are known functions, then we can solve the following linear parabolic equation with the smooth initial data $\theta_0^{\epsilon}$
\begin{equation}\label{vsmooth}
\left\{\begin{array}{ll}
\theta_t+(\bm{u}^{\epsilon}\cdot\nabla)\theta-\nabla\cdot(\kappa(f^{\epsilon})\nabla\theta)=-\nabla\cdot \bm{v}^{\epsilon},\\
\theta|_{\p\Omega}=0,\\
\theta(\bm{x},0)=\theta_0^{\epsilon},
\end{array}\right.
\end{equation}
and denote the solution by $\theta^{\epsilon}$. Based on the above results, we can define the mapping as
\beno
F^{\epsilon}(f)=\theta^{\epsilon}.
\eeno

To make use of Schauder's fixed point theorem, the next step is to verify that $F^{\epsilon}$ meets the conditions that for any fixed $\epsilon\in(0,1), F^{\epsilon}:D\rightarrow D$ is continuous and compact. Specifically, we need to confirm
\begin{enumerate}
\item[(a)] $\|\theta^{\epsilon}\|_{X}\leq R_0,\,\,\forall f\in D$;
\item[(b)] $\|\theta^{\epsilon}\|_{C(0,T;H_0^1)}+\|\theta^{\epsilon}\|_{L^2(0,T; H^2)}\leq C,\,\,\forall f\in D$;
\item[(c)]For any $\eta>0$, there exists $\delta=\delta(\eta)>0$ such that for any $f_1, f_2\in D$ with $\|f_1-f_2\|_X<\delta$, it follows that $\|F^{\epsilon}(f_1)-F^{\epsilon}(f_2)\|_X<\eta.$
\end{enumerate}

We verify (a) first. Taking inner product of the first two equations of \eqref{molify} with $(\bm{u}, \bm{v})$ respectively, by integration by parts, it leads to
\beno
&&\frac{1}{2}\frac{{\rm d}}{{\rm d}t}(\|\bm{u}(t)\|_{L^2}^2+\|\bm{v}(t)\|_{L^2}^2)
+\frac{1}{\sigma}(\|\nabla\bm{u}\|_{L^2}^2+\|\nabla\bm{v}\|_{L^2}^2)\\
&\leq&-\int_\Omega\nabla\cdot(\bm{v}\otimes\bm{v})\cdot\bm{u}{\rm d}\bm{x}-\int_\Omega(\bm{v}\cdot\nabla)\bm{u}\cdot\bm{v}{\rm d}\bm{x}-\int_\Omega\nabla f^{\epsilon}\cdot\bm{v}{\rm d}\bm{x}\\
&\leq&\frac{1}{2\sigma}\|\nabla\bm{v}\|_{L^2}^2+C\|f^{\epsilon}\|_{L^2}^2,
\eeno
that is,
\ben\label{aprove}
\frac{{\rm d}}{{\rm d}t}(\|\bm{u}(t)\|_{L^2}^2+\|\bm{v}(t)\|_{L^2}^2)
+\frac{2}{\sigma}\|\nabla\bm{u}\|_{L^2}^2+\frac{1}{\sigma}\|\nabla\bm{v}\|_{L^2}^2
\leq C\|f^{\epsilon}\|_{L^2}^2.
\een
After integration over $[0,T]$ for \eqref{aprove} , we attain
\ben\label{aprove1}
&&\|\bm{u}^{\epsilon}\|_{L^2}^2+\|\bm{v}^{\epsilon}\|_{L^2}^2
+\frac{2}{\sigma}\int_0^T\|\nabla\bm{u}^{\epsilon}\|_{L^2}^2{\rm d}t+\frac{1}{\sigma}\int_0^T\|\nabla\bm{v}^{\epsilon}\|_{L^2}^2{\rm d}t\notag\\
&\leq&\|\bm{u}_0^{\epsilon}\|_{L^2}^2+\|\bm{v}_0^{\epsilon}\|_{L^2}^2+C\int_0^T\|f^{\epsilon}\|_{L^2}^2{\rm d}t\notag\\
&\leq&\|\bm{u}_0\|_{L^2}^2+\|\bm{v}_0\|_{L^2}^2+CTR_0^2.
\een

Multiplying the first equation of $\eqref{vsmooth}$ by $\theta$, it follows
\beno
\frac{{\rm d}}{{\rm d}t}\|\theta(t)\|_{L^2}^2+\frac{1}{\sigma}\|\nabla\theta\|_{L^2}^2
\leq C\|\bm{v}^{\epsilon}\|_{L^2}^2,
\eeno
After integration over $[0,T]$ for \eqref{aprove} , we derive
\ben\label{aprove3}
&&\|\theta(t)\|_{L^2}^2+\frac{1}{\sigma}\int_0^T\|\nabla\theta\|_{L^2}^2{\rm d}t
\,\leq\|\theta_0\|_{L^2}^2+C\int_0^T\|\bm{v}^{\epsilon}\|_{L^2}^2{\rm d}t\notag\\
&\leq&\|\theta_0\|_{L^2}^2+CT(\|\bm{u}_0\|_{L^2}^2+\|\theta_0\|_{L^2}^2+CTR_0^2).
\een

To show that $F^\epsilon$ maps $D$ to $D$, it suffices to certify that the
right-hand side of \eqref{aprove3} is bounded by $R_0^2$.
Then we obtain a condition for $R_0$, namely
\beno
\|\theta_0\|_{L^2}^2+CT(\|\bm{u}_0\|_{L^2}^2+\|\theta_0\|_{L^2}^2+CTR_0^2)\leq R_0^2.
\eeno
Thus, provided that $T$ is sufficiently small, such that $CT^2\ll\frac{1}{2}$, the above inequality would hold. Similarly, we can deduce (b) and (c) when $T$ is sufficiently small.
Schauder's fixed point theorem gives us access to the existence of a solution on a finite time interval $T$. The uniform estimates in Proposition \ref{weak1}
would enable us to pass the limit to attain a weak solution $(\bm{u},\bm{v},\theta)$.

Based on the global bounds obtained in Proposition \ref{weak1}, the local solution obtained by Schauder's fixed point theorem can be easily extended into a global solution via
Picard type extension theorem. As a result, we are able to receive the desired
global weak solutions.
\end{proof}

\section{UNIQUE GLOBAL STRONG SOLUTION}
\label{proofT1}
\setcounter{equation}{0}
This section is devoted to acquiring the strong solution and its uniqueness to the system \eqref{3D}-\eqref{eq20}.

\begin{proposition}\label{sita-H2}
Let $\Omega\subset\mathbb{R}^2$ be a bounded domain with smooth boundary. Assuming that $\theta_0\in H_0^1\cap H^2,\\(\bm{u}_0,\bm{v}_0)\in H_0^1$, then for any solution $(\bm{u},\bm{v},\theta)$ of the system $\eqref{3D}-\eqref{eq20}$, for any $t>0$, it holds
\beno
\|\nabla^2\theta\|_{L^2}^2+\|\theta_t\|_{L^2}^2\leq Ce^{-\alpha t}
\eeno
and
\beno
\int_0^te^{\alpha\tau}\|\nabla\theta_\tau\|_{L^2}^2{\rm d}\tau+\int_0^te^{\alpha\tau}\|\theta\|_{H^3}^2{\rm d}\tau
\leq C,
\eeno
where $C=C(\|\theta_0\|_{H^2},\|\bm{u}_0\|_{H^1},\|\bm{v}_0\|_{H^1},\sigma,\alpha).$
\end{proposition}
\begin{proof}
Taking the derivative of time to the third equation of $\eqref{3D}$ gives that
\ben\label{H2-1}
\partial_t\theta_t+(\bm{u}\cdot\nabla)\theta_t-\nabla\cdot(\kappa(\theta)\nabla\theta_t)
=\nabla\cdot(\kappa'(\theta)\theta_t\nabla\theta)-\bm{u}_t\cdot\nabla\theta-\nabla\cdot\bm{v}_t.
\een
Multiplying by $\theta_t$ and exploiting \eqref{yita-22}, Lemma \ref{P1}, we derive
\ben\label{H2-2}
&&\frac{1}{2}\frac{{\rm d}}{{\rm d}t}\|\theta_t\|_{L^2}^2+\frac{1}{\sigma}\|\nabla\theta_t\|_{L^2}^2\notag\\
&\leq&-\int_\Omega\kappa'(\theta)\theta_t\nabla\theta\cdot\nabla\theta_t{\rm d}\bm{x}
-\int_\Omega\bm{u}_t\cdot\nabla\theta\theta_t{\rm d}\bm{x}-\int_\Omega(\nabla\cdot\bm{v}_t)\theta_t{\rm d}\bm{x}\notag\\
&\leq&\widetilde{M}\|\theta_t\|_{L^4}\|\nabla\theta\|_{L^4}\|\nabla\theta_t\|_{L^2}
+C\|\bm{u}_t\|_{L^2}\|\nabla\theta\|_{L^4}\|\theta_t\|_{L^4}
+C\|\bm{v}_t\|_{L^2}\|\nabla\theta_t\|_{L^2}\notag\\
&\leq&C\|\theta_t\|_{L^2}^{\frac{1}{2}}\|\nabla\theta_t\|_{L^2}^{\frac{3}{2}}\|\nabla\theta\|_{L^2}^{\frac{1}{2}}
\|\nabla^2\theta\|_{L^2}^{\frac{1}{2}}
+C\|\theta_t\|_{L^2}^{\frac{1}{2}}\|\nabla\theta_t\|_{L^2}^{\frac{3}{2}}\|\nabla\theta\|_{L^2}\notag\\
&&+C\|\bm{u}_t\|_{L^2}\|\theta_t\|_{L^2}^{\frac{1}{2}}\|\nabla\theta_t\|_{L^2}^{\frac{1}{2}}
\|\nabla\theta\|_{L^2}^{\frac{1}{2}}\|\nabla^2\theta\|_{L^2}^{\frac{1}{2}}
+C\|\bm{u}_t\|_{L^2}\|\theta_t\|_{L^2}^{\frac{1}{2}}\|\nabla\theta_t\|_{L^2}^{\frac{1}{2}}\|\nabla\theta\|_{L^2}\notag\\
&&+C\|\bm{v}_t\|_{L^2}\|\nabla\theta_t\|_{L^2}\notag\\
&\leq&\frac{1}{2\sigma}\|\nabla\theta_t\|_{L^2}^2+C\|\bm{u}_t\|_{L^2}^2+C\|\bm{v}_t\|_{L^2}^2
+C\|\theta_t\|_{L^2}^2(\|\nabla\theta\|_{L^2}^2\|\nabla^2\theta\|_{L^2}^2+\|\nabla\theta\|_{L^2}^4).
\een
By means of Proposition \ref{seta-L2H2} and Proposition \ref{uv-H2}, multiplying $e^{\alpha t}$ with \eqref{H2-2} and using the Gronwall inequality, we conclude
\ben\label{H2-3}
&&e^{\alpha t}\|\theta_t\|_{L^2}^2+\frac{1}{\sigma}\int_0^t e^{\alpha\tau}\|\nabla\theta_\tau\|_{L^2}^2{\rm d}\tau\notag\\
&\leq&C\Big[\|\theta_t(0)\|_{L^2}^2+\int_0^t e^{\alpha\tau}
(\|\bm{u}_\tau\|_{L^2}^2+\|\bm{v}_\tau\|_{L^2}^2+\|\theta_\tau\|_{L^2}^2){\rm d}\tau\Big] \exp\Big[{\int_0^t(\|\nabla\theta\|_{L^2}^2\|\nabla^2\theta\|_{L^2}^2+\|\nabla\theta\|_{L^2}^4){\rm d}\tau}\Big]\notag\\
&\leq&C,
\een
here the value of $\|\theta_t(0)\|_{L^2}$ will be controlled by $\|\theta_0\|_{H^2}$.
By virtue of \eqref{da-seta8}, \eqref{H2-5} and Lemma \ref{setacompare}, it is easy to get
\beno
\|\nabla^2\theta\|_{L^2}^2\leq Ce^{-\alpha t}.
\eeno

We update the third equation of $\eqref{3D}$ as
\beno
-\kappa(\theta)\Delta\theta=\kappa'(\theta)\nabla\theta\cdot\nabla\theta-\theta_t-(\bm{u}\cdot\nabla)\theta-\nabla\cdot\bm{v}.
\eeno
Applying Lemma \ref{wup}, Proposition $\ref{weak2}-\ref{uv-H2}$ shows that
\beno
\|\theta\|_{H^3}
&\leq&C\|\kappa'(\theta)\nabla\theta\cdot\nabla\theta\|_{H^1}+C\|\theta_t\|_{H^1}
+C\|\bm{u}\cdot\nabla\theta\|_{H^1}+C\|\nabla\cdot\bm{v}\|_{H^1}\\
&\leq&C\|\nabla\theta\|_{L^4}^2+C\|\nabla^2\theta\|_{L^4}\|\nabla\theta\|_{L^4}+C\|\theta_t\|_{H^1}
+C\|\bm{u}\|_{L^4}\|\nabla\theta\|_{L^4}\\
&&+C\|\nabla\bm{u}\|_{L^4}\|\nabla\theta\|_{L^4}+C\|\bm{u}\|_{L^4}\|\nabla^2\theta\|_{L^4}+C\|\nabla\cdot\bm{v}\|_{H^1}\\
&\leq&Ce^{-\alpha t}+C\|\theta_t\|_{H^1}+C\|\nabla\bm{v}\|_{H^1}+C\|\bm{u}\|_{L^4}\|\nabla^2\theta\|_{L^4}\\
&&+C(\|\nabla\theta\|_{L^2}^\frac{1}{2}\|\nabla^2\theta\|_{L^2}^\frac{1}{2}+\|\nabla\theta\|_{L^2})(\|\nabla^2\theta\|_{L^4}+\|\nabla\bm{u}\|_{L^4})\\
&\leq&Ce^{-\alpha t}+C\|\theta_t\|_{H^1}+C\|\nabla\bm{v}\|_{H^1}
+Ce^{-\frac{\alpha t}{2}}
(\|\nabla^2\theta\|_{L^2}^\frac{1}{2}\|\nabla^3\theta\|_{L^2}^\frac{1}{2}+\|\nabla^2\theta\|_{L^2}\\
&&+\|\nabla\bm{u}\|_{L^2}^\frac{1}{2}\|\nabla^2\bm{u}\|_{L^2}^\frac{1}{2}+\|\nabla\bm{u}\|_{L^2})\\
&\leq&\frac{1}{2}\|\nabla^3\theta\|_{L^2}+Ce^{-\alpha t}+C\|\theta_t\|_{H^1}+C\|\nabla\bm{v}\|_{H^1}+C\|\nabla^2\bm{u}\|_{L^2},
\eeno
which easily derives $\int_0^te^{\alpha\tau}\|\theta\|_{H^3}^2{\rm d}\tau\leq C.$
\end{proof}
\begin{proposition}\label{uv-qiangH2}
Under the assumptions of Proposition \ref{sita-H2}, we further suppose $(\bm{u}_0,\bm{v}_0)\in H^2$, then for any solution $(\bm{u},\bm{v},\theta)$ of the system $\eqref{3D}-\eqref{eq20}$, for any $t>0$, we hold
\beno
\|\nabla^2\bm{u}\|_{L^2}^2+\|\nabla^2\bm{v}\|_{L^2}^2+\|\bm{u}_t\|_{L^2}^2+\|\bm{v}_t\|_{L^2}^2
\leq Ce^{-\alpha t}
\eeno
and
\beno
\int_0^te^{\alpha\tau}(\|\nabla\bm{u}_\tau\|_{L^2}^2+\|\nabla\bm{v}_\tau\|_{L^2}^2){\rm d}\tau+\int_0^te^{\alpha\tau}(\|\bm{u}\|_{H^3}^2+\|\bm{v}\|_{H^3}^2){\rm d}\tau
\leq C,
\eeno
where $C=C(\|\theta_0\|_{H^2},\|\bm{u}_0\|_{H^2},\|\bm{v}_0\|_{H^2},\sigma,\alpha).$
\end{proposition}
\begin{proof}
Taking the temporal derivatives of the first two equations in $\eqref{3D}$ contributes to
\begin{eqnarray}
\left\{\begin{array}{ll}
\partial_t\bm{u}_t+(\bm{u}\cdot\nabla)\bm{u}_t-\nabla\cdot(\mu(\theta)\nabla\bm{u}_t)+\nabla p_t=\nabla\cdot(\mu'(\theta)\theta_t\nabla\bm{u})-\bm{u}_t\cdot\nabla\bm{u}-\nabla\cdot(\bm{v}\otimes\bm{v})_t,\\
\partial_t\bm{v}_t+(\bm{u}\cdot\nabla)\bm{v}_t-\nabla\cdot(\nu(\theta)\nabla\bm{v}_t) =\nabla\cdot(\nu'(\theta)\theta_t\nabla\bm{v})-\bm{u}_t\cdot\nabla\bm{v}-\bm{v}_t\cdot\nabla\bm{u}-\bm{v}\cdot\nabla\bm{u}_t-\nabla\theta_t.
\end{array}\right.
\end{eqnarray}
Taking inner product with $\bm{u}_t$ and $\bm{v}_t$ respectively, one can attain
\ben\label{H21}
&&\frac{1}{2}\frac{{\rm d}}{{\rm d}t}(\|\bm{u}_t\|_{L^2}^2+\|\bm{v}_t\|_{L^2}^2)
+\frac{1}{\sigma}(\|\nabla\bm{u}_t\|_{L^2}^2+\|\nabla\bm{v}_t\|_{L^2}^2)\notag\\
&\leq&-\int_\Omega\mu'(\theta)\theta_t\nabla\bm{u}\cdot\nabla\bm{u}_t{\rm d}\bm{x}
-\int_\Omega\nu'(\theta)\theta_t\nabla\bm{v}\cdot\nabla\bm{v}_t{\rm d}\bm{x}
-\int_\Omega\bm{u}_t\cdot\nabla\bm{u}\cdot\bm{u}_t{\rm d}\bm{x}-\int_\Omega\bm{u}_t\cdot\nabla\bm{v}\cdot\bm{v}_t{\rm d}\bm{x}\notag\\
&&-\int_\Omega\nabla\cdot(\bm{v}\otimes\bm{v})_t\cdot\bm{u}_t{\rm d}\bm{x}
-\int_\Omega\bm{v}_t\cdot\nabla\bm{u}\cdot\bm{v}_t{\rm d}\bm{x}
-\int_\Omega\bm{v}\cdot\nabla\bm{u}_t\cdot\bm{v}_t{\rm d}\bm{x}
-\int_\Omega\nabla\theta_t\cdot\bm{v}_t{\rm d}\bm{x}\notag\\
&=&\sum\limits_{i=1}^{8}I_i.
\een
Owing to \eqref{yita-22},
\ben\label{H22}
I_1+I_2
&\leq&C\|\theta_t\|_{L^4}(\|\nabla\bm{u}\|_{L^4}\|\nabla\bm{u}_t\|_{L^2}
+\|\nabla\bm{v}\|_{L^4}\|\nabla\bm{v}_t\|_{L^2})\notag\\
&\leq&\frac{1}{4\sigma}(\|\nabla\bm{u}_t\|_{L^2}^2+\|\nabla\bm{v}_t\|_{L^2}^2)
+C\|\theta_t\|_{L^4}^2(\|\nabla\bm{u}\|_{L^4}^2+\|\nabla\bm{v}\|_{L^4}^2)\notag\\
&\leq&\frac{1}{4\sigma}(\|\nabla\bm{u}_t\|_{L^2}^2+\|\nabla\bm{v}_t\|_{L^2}^2)
+C\|\theta_t\|_{L^2}\|\nabla\theta_t\|_{L^2}(\|\nabla\bm{u}\|_{L^2}\|\nabla^2\bm{u}\|_{L^2}\notag\\
&&+\|\nabla\bm{v}\|_{L^2}\|\nabla^2\bm{v}\|_{L^2}+\|\nabla\bm{u}\|_{L^2}^2+\|\nabla\bm{v}\|_{L^2}^2)\notag\\
&\leq&\frac{1}{4\sigma}(\|\nabla\bm{u}_t\|_{L^2}^2+\|\nabla\bm{v}_t\|_{L^2}^2)
+C\|\theta_t\|_{L^2}^2\|\nabla\theta_t\|_{L^2}^2
+C(\|\nabla\bm{u}\|_{L^2}^2\|\nabla^2\bm{u}\|_{L^2}^2\notag\\
&&+\|\nabla\bm{v}\|_{L^2}^2\|\nabla^2\bm{v}\|_{L^2}^2+\|\nabla\bm{u}\|_{L^2}^4+\|\nabla\bm{v}\|_{L^2}^4).
\een
In addition,
\ben\label{H23}
\sum\limits_{i=3}^{8}I_i
&\leq&C(\|\bm{u}_t\|_{L^4}^2\|\nabla\bm{u}\|_{L^2}+\|\bm{u}_t\|_{L^4}\|\bm{v}_t\|_{L^4}\|\nabla\bm{v}\|_{L^2}
+\|\nabla\bm{u}_t\|_{L^2}\|\bm{v}\|_{L^4}\|\bm{v}_t\|_{L^4}\notag\\
&&+\|\nabla\bm{u}\|_{L^2}\|\bm{v}_t\|_{L^4}^2+\|\theta_t\|_{L^2}\|\nabla\bm{v}_t\|_{L^2})\notag\\
&\leq&C(\|\bm{u}_t\|_{L^2}\|\nabla\bm{u}_t\|_{L^2}+\|\bm{v}_t\|_{L^2}\|\nabla\bm{v}_t\|_{L^2})
(\|\nabla\bm{u}\|_{L^2}+\|\nabla\bm{v}\|_{L^2})\notag\\
&&+C\|\nabla\bm{v}_t\|_{L^2}\|\theta_t\|_{L^2}+C\|\bm{v}\|_{L^2}^{\frac{1}{2}}\|\nabla\bm{v}\|_{L^2}^{\frac{1}{2}}
\|\bm{v}_t\|_{L^2}^{\frac{1}{2}}\|\nabla\bm{v}_t\|_{L^2}^{\frac{1}{2}}\|\nabla\bm{u}_t\|_{L^2}\notag\\
&\leq&\frac{1}{4\sigma}(\|\nabla\bm{u}_t\|_{L^2}^2+\|\nabla\bm{v}_t\|_{L^2}^2)
+C(\|\bm{u}_t\|_{L^2}^2+\|\bm{v}_t\|_{L^2}^2)(\|\nabla\bm{u}\|_{L^2}^2+\|\nabla\bm{v}\|_{L^2}^2)\notag\\
&&+C\|\theta_t\|_{L^2}^2.
\een
Substituting \eqref{H22} and \eqref{H23} into \eqref{H21}, it follows that
\ben\label{H24}
&&\frac{{\rm d}}{{\rm d}t}(\|\bm{u}_t\|_{L^2}^2+\|\bm{v}_t\|_{L^2}^2)
+\frac{1}{\sigma}(\|\nabla\bm{u}_t\|_{L^2}^2+\|\nabla\bm{v}_t\|_{L^2}^2)\notag\\
&\leq&C(\|\bm{u}_t\|_{L^2}^2+\|\bm{v}_t\|_{L^2}^2)(\|\nabla\bm{u}\|_{L^2}^2+\|\nabla\bm{v}\|_{L^2}^2)
+C\|\theta_t\|_{L^2}^2\|\nabla\theta_t\|_{L^2}^2\notag\\
&&+\,C(\|\nabla\bm{u}\|_{L^2}^2\|\nabla^2\bm{u}\|_{L^2}^2
+\|\nabla\bm{v}\|_{L^2}^2\|\nabla^2\bm{v}\|_{L^2}^2+\|\nabla\bm{u}\|_{L^2}^4+\|\nabla\bm{v}\|_{L^2}^4)
+C\|\theta_t\|_{L^2}^2.
\een
Multiplying $e^{\alpha t}$ with \eqref{H24} and using the Gronwall inequality, it is accessible to get
\ben\label{H25}
&&e^{\alpha t}(\|\bm{u}_t\|_{L^2}^2+\|\bm{v}_t\|_{L^2}^2)
+\frac{1}{\sigma}\int_0^te^{\alpha\tau}(\|\nabla\bm{u}_\tau\|_{L^2}^2+\|\nabla\bm{v}_\tau\|_{L^2}^2){\rm d}\tau\notag\\
&\leq&C\Big[\|\bm{u}_t(0)\|_{L^2}^2+\|\bm{v}_t(0)\|_{L^2}^2+\int_0^te^{\alpha\tau}
(\|\bm{u}_\tau\|_{L^2}^2+\|\bm{v}_\tau\|_{L^2}^2+\|\theta_\tau\|_{L^2}^2\|\nabla\theta_\tau\|_{L^2}^2+\|\nabla\bm{u}\|_{L^2}^2\|\nabla^2\bm{u}\|_{L^2}^2\notag\\
&&+\|\nabla\bm{v}\|_{L^2}^2\|\nabla^2\bm{v}\|_{L^2}^2+\|\nabla\bm{u}\|_{L^2}^4+\|\nabla\bm{v}\|_{L^2}^4+\|\theta_\tau\|_{L^2}^2){\rm d}\tau\Big]\exp\Big[{\int_0^t(\|\nabla\bm{u}\|_{L^2}^2+\|\nabla\bm{v}\|_{L^2}^2){\rm d}\tau}\Big]\notag\\
&\leq&C.
\een
According to \eqref{uvH29} and \eqref{uvH211}, we know
\ben\label{H26}
\|\nabla^2\bm{u}\|_{L^2}^2+\|\nabla^2\bm{v}\|_{L^2}^2
\leq Ce^{-\alpha t}.
\een
\vskip .1in
Applying Corollary \ref{C1}, \eqref{pL2} and Lemma \ref{wup}, we have
\beno
&&\|\bm{u}\|_{H^3}+\|\bm{v}\|_{H^3}+\|\nabla^2p\|_{L^2}\\
&\leq&C(\|\bm{u}_t\|_{H^1}+\|\bm{u}\cdot\nabla\bm{u}\|_{H^1}+\|\mu'(\theta)\nabla\theta\cdot\nabla\bm{u}\|_{H^1}
+\|\nabla\cdot(\bm{v}\otimes\bm{v})\|_{H^1}+\|\bm{u}\|_{H^2}+\|\nu'(\theta)\nabla\theta\cdot\nabla\bm{v}\|_{H^1}\\
&&+\|\bm{v}_t\|_{H^1}+\|\bm{u}\cdot\nabla\bm{v}\|_{H^1}
+\|\nabla\theta\|_{H^1}+\|\bm{v}\cdot\nabla\bm{u}\|_{H^1})\\
&\leq&C(\|\bm{u}_t\|_{H^1}+\|\bm{v}_t\|_{H^1}+\|\bm{u}\|_{H^2}+\|\nabla\theta\|_{H^1})+C(\|\bm{u}\|_{L^4}\|\nabla\bm{u}\|_{L^4}+\|\nabla\bm{u}\|_{L^4}^2+\|\bm{u}\|_{L^4}\|\nabla^2\bm{u}\|_{L^4}\\
&&+\|\nabla\theta\|_{L^4}\|\nabla\bm{u}\|_{L^4}+\|\nabla^2\theta\|_{L^4}\|\nabla\bm{u}\|_{L^4}
+\|\nabla\theta\|_{L^4}\|\nabla^2\bm{u}\|_{L^4}+\|\nabla\bm{v}\|_{L^4}\|\bm{v}\|_{L^4}+\|\nabla^2\bm{v}\|_{L^4}\|\bm{v}\|_{L^4}+\|\nabla\bm{v}\|_{L^4}^2\\
&&+\|\nabla\theta\|_{L^4}\|\nabla\bm{v}\|_{L^4}+\|\nabla^2\theta\|_{L^4}\|\nabla\bm{v}\|_{L^4}+\|\nabla\theta\|_{L^4}\|\nabla^2\bm{v}\|_{L^4}+\|\bm{u}\|_{L^4}\|\nabla\bm{v}\|_{L^4}+\|\nabla\bm{u}\|_{L^4}\|\nabla\bm{v}\|_{L^4}\\
&&+\|\bm{u}\|_{L^4}\|\nabla^2\bm{v}\|_{L^4}+\|\bm{v}\|_{L^4}\|\nabla\bm{u}\|_{L^4}+\|\bm{v}\|_{L^4}\|\nabla^2\bm{u}\|_{L^4})\\
&\leq&C(\|\bm{u}_t\|_{H^1}+\|\bm{v}_t\|_{H^1}+\|\bm{u}\|_{H^2}+\|\nabla\theta\|_{H^1})
+Ce^{-\frac{\alpha t}{2}}(\|\nabla^2\bm{u}\|_{L^4}+\|\nabla^2\bm{v}\|_{L^4})\\
&&+C\|\nabla^2\theta\|_{L^4}(\|\nabla\bm{u}\|_{L^4}+\|\nabla\bm{v}\|_{L^4})+Ce^{-\alpha t}\\
&\leq&C(\|\bm{u}_t\|_{H^1}+\|\bm{v}_t\|_{H^1}+\|\bm{u}\|_{H^2}+\|\nabla\theta\|_{H^1})
+Ce^{-\frac{\alpha t}{2}}(\|\nabla^2\bm{u}\|_{L^2}^{\frac{1}{2}}\|\nabla^3\bm{u}\|_{L^2}^{\frac{1}{2}}\\
&&+\|\nabla^2\bm{v}\|_{L^2}^{\frac{1}{2}}\|\nabla^3\bm{v}\|_{L^2}^{\frac{1}{2}}+\|\nabla^2\bm{u}\|_{L^2}
+\|\nabla^2\bm{v}\|_{L^2})+C(\|\nabla^2\theta\|_{L^2}^{\frac{1}{2}}\|\nabla^3\theta\|_{L^2}^{\frac{1}{2}}
+\|\nabla^2\theta\|_{L^2})\\
&&\times(\|\nabla\bm{u}\|_{L^2}^{\frac{1}{2}}\|\nabla^2\bm{u}\|_{L^2}^{\frac{1}{2}}
+\|\nabla\bm{v}\|_{L^2}^{\frac{1}{2}}\|\nabla^2\bm{v}\|_{L^2}^{\frac{1}{2}}
+\|\nabla\bm{u}\|_{L^2}+\|\nabla\bm{v}\|_{L^2})+Ce^{-\alpha t}\\
&\leq&\frac{1}{2}(\|\nabla^3\bm{u}\|_{L^2}+\|\nabla^3\bm{v}\|_{L^2})
+C(\|\bm{u}_t\|_{H^1}+\|\bm{v}_t\|_{H^1}+\|\bm{u}\|_{H^2}+\|\nabla\theta\|_{H^1})+Ce^{-\alpha t}+C\|\nabla^3\theta\|_{L^2},
\eeno
which is responsible for
$$\int_0^te^{\alpha\tau}(\|\bm{u}\|_{H^3}^2+\|\bm{v}\|_{H^3}^2){\rm d}\tau
\leq C.$$
\end{proof}
\begin{proof}[\bf Proof of Part \textbf{(b)} of Theorem \ref{T1}:]
The existence of the strong solution can be finished from Part \textbf{(a)} of Theorem \ref{T1} and Proposition \ref{sita-H2}-\ref{uv-qiangH2}. So the left part is to build up the uniqueness of strong solution.

\noindent $\it{Uniqueness}$:
Assume $(\hat{\bm{u}}, \hat{\bm{v}}, \hat{\theta}, \hat{p})$ and $(\widetilde{\bm{u}}, \widetilde{\bm{v}}, \widetilde{\theta}, \widetilde{p})$  are two solutions of the system $\eqref{3D}-\eqref{eq20}$. We denote
$$\mathbf{U}=\hat{\bm{u}}-\widetilde{\bm{u}},\,\,\mathbf{V}=\hat{\bm{v}}-\widetilde{\bm{v}},\,\Theta=\hat{\theta}-\widetilde{\theta},\,\,P=\hat{p}-\widetilde{p},$$ which solves the following initial-boundary value problem
\begin{equation}\label{uni}
\left\{\begin{array}{ll}
\mathbf{U}_t+(\hat{\bm{u}}\cdot\nabla)\mathbf{U}-\nabla\cdot(\mu(\hat{\theta})\nabla\hat{\bm{u}}-\mu(\widetilde{\theta})\nabla\widetilde{\bm{u}})+\nabla P=\alpha_1,\\
\mathbf{V}_t+(\hat{\bm{u}}\cdot\nabla)\mathbf{V}-\nabla\cdot(\nu(\hat{\theta})\nabla\hat{\bm{v}}-\nu(\widetilde{\theta})\nabla\widetilde{\bm{v}})=\alpha_2,\\
\Theta_t+(\hat{\bm{u}}\cdot\nabla)\Theta-\nabla\cdot(\kappa(\hat{\theta})\nabla\hat{\theta}-\kappa(\widetilde{\theta})\nabla\widetilde{\theta})=\alpha_3,\\
\nabla\cdot\mathbf{U}=0,\,\,\mathbf{U}|_{\p\Omega}=\mathbf{V}|_{\p\Omega}=\Theta|_{\p\Omega}=0,\\
(\mathbf{U}, \mathbf{V}, \Theta)(\bm{x},0)=0,
\end{array}\right.
\end{equation}
where
\beno
&&\alpha_1=-(\mathbf{U}\cdot\nabla)\widetilde{\bm{u}}-\nabla\cdot(\mathbf{V}\otimes\hat{\bm{v}})-\nabla\cdot(\widetilde{\bm{v}}\otimes\mathbf{V}),\\
&&\alpha_2=-(\mathbf{U}\cdot\nabla)\widetilde{\bm{v}}-(\hat{\bm{v}}\cdot\nabla)\mathbf{U}-(\mathbf{V}\cdot\nabla)\widetilde{\bm{u}}-\nabla\Theta,\\
&&\alpha_3=-(\mathbf{U}\cdot\nabla)\widetilde{\theta}-\nabla\cdot\mathbf{V}.
\eeno
Taking inner product with \eqref{uni} by $(\mathbf{U},\mathbf{V}, \Theta)$ respectively yields
\ben\label{uni-est}
&&\f12\f{{\rm d}}{{\rm d}t}(\|\mathbf{U}\|_{L^2}^2+\|\mathbf{V}\|_{L^2}^2+\|\Theta\|_{L^2}^2)
-\int_\Omega\nabla\cdot(\mu(\hat{\theta})\nabla\hat{\bm{u}}-\mu(\widetilde{\theta})\nabla\widetilde{\bm{u}})\cdot\mathbf{U}{\rm d}\bm{x}\notag\\
&&-\int_\Omega\nabla\cdot(\nu(\hat{\theta})\nabla\hat{\bm{v}}-\nu(\widetilde{\theta})\nabla\widetilde{\bm{v}})\cdot\mathbf{V}{\rm d}\bm{x}
-\int_\Omega\nabla\cdot(\kappa(\hat{\theta})\nabla\hat{\theta}-\kappa(\widetilde{\theta})\nabla\widetilde{\theta})\Theta {\rm d}\bm{x}\notag\\
&=&\int_\Omega\alpha_1\cdot\mathbf{U}{\rm d}\bm{x}+\int_\Omega\alpha_2\cdot\mathbf{V}{\rm d}\bm{x}+\int_\Omega\alpha_3\Theta {\rm d}\bm{x}\notag\\
&\leq&C\|\nabla\widetilde{\bm{u}}\|_{L^2}\|\mathbf{U}\|_{L^4}^2
+C\|\widetilde{\bm{v}}\|_{L^\infty}\|\mathbf{V}\|_{L^2}\|\nabla\mathbf{U}\|_{L^2}
+C\|\hat{\bm{v}}\|_{L^\infty}\|\mathbf{V}\|_{L^2}\|\nabla\mathbf{U}\|_{L^2}\notag\\
&&+\,C\|\nabla\widetilde{\bm{v}}\|_{L^2}\|\mathbf{V}\|_{L^4}\|\mathbf{U}\|_{L^4}
+C\|\nabla\widetilde{\bm{u}}\|_{L^2}\|\mathbf{V}\|_{L^4}^2
+C\|\nabla\widetilde{\theta}\|_{L^2}\|\Theta\|_{L^4}\|\mathbf{U}\|_{L^4}\notag\\
&\leq&C\|\nabla\widetilde{\bm{u}}\|_{L^2}\|\mathbf{U}\|_{L^4}^2
+C\|\widetilde{\bm{v}}\|_{H^2}\|\mathbf{V}\|_{L^2}\|\nabla\mathbf{U}\|_{L^2}
+C\|\hat{\bm{v}}\|_{H^2}\|\mathbf{V}\|_{L^2}\|\nabla\mathbf{U}\|_{L^2}\notag\\
&&+\,C\|\nabla\widetilde{\bm{v}}\|_{L^2}\|\mathbf{V}\|_{L^4}\|\mathbf{U}\|_{L^4}
+C\|\nabla\widetilde{\bm{u}}\|_{L^2}\|\mathbf{V}\|_{L^4}^2
+C\|\nabla\widetilde{\theta}\|_{L^2}\|\Theta\|_{L^4}\|\mathbf{U}\|_{L^4}\notag\\
&\leq&C(\|\mathbf{U}\|_{L^2}\|\nabla\mathbf{U}\|_{L^2}+\|\mathbf{V}\|_{L^2}\|\nabla\mathbf{V}\|_{L^2}+\|\Theta\|_{L^2}\|\nabla\Theta\|_{L^2})
(\|\nabla\widetilde{\bm{u}}\|_{L^2}+\|\nabla\widetilde{\bm{v}}\|_{L^2}+\|\nabla\widetilde{\theta}\|_{L^2})\notag\\
&&+\,C(\|\widetilde{\bm{v}}\|_{H^2}+\|\hat{\bm{v}}\|_{H^2})\|\mathbf{V}\|_{L^2}\|\nabla\mathbf{U}\|_{L^2}\notag\\
&\leq&\frac{1}{2\sigma}(\|\nabla\mathbf{U}\|_{L^2}^2+\|\nabla\mathbf{V}\|_{L^2}^2+\|\nabla\Theta\|_{L^2}^2)
+Ce^{-\alpha t}(\|\mathbf{U}\|_{L^2}^2+\|\mathbf{V}\|_{L^2}^2+\|\Theta\|_{L^2}^2),
\een
here we use $\|Y\|_{L^\infty}\leq C\|Y\|_{H^2}$.
Now we concentrate on dealing with the left term of the equation \eqref{uni-est}. By integration by parts and the boundary condition $\mathbf{U}|_{\p\Omega}=0$, we attain
\beno
&&-\int_\Omega\nabla\cdot(\mu(\hat{\theta})\nabla\hat{\bm{u}}-\mu(\widetilde{\theta})\nabla\widetilde{\bm{u}})\cdot\mathbf{U}{\rm d}\bm{x}
=-\int_\Omega\partial_i(\mu(\hat{\theta})\partial_i\hat{u}_j-\mu(\widetilde{\theta})\partial_i\widetilde{u}_j)U_j{\rm d}\bm{x}\\
&=&\int_\Omega(\mu(\hat{\theta})\partial_i\hat{u}_j-\mu(\widetilde{\theta})\partial_i\widetilde{u}_j)\partial_iU_j{\rm d}\bm{x}\\
&=&\int_\Omega\mu(\hat{\theta})\partial_iU_j\partial_iU_j{\rm d}\bm{x}+\int_\Omega(\mu(\hat{\theta})-\mu(\widetilde{\theta}))\partial_i\widetilde{u}_j\partial_iU_j{\rm d}\bm{x}\\
&\geq&\frac{1}{\sigma}\|\nabla\mathbf{U}\|_{L^2}^2+\int_\Omega(\mu(\hat{\theta})-\mu(\widetilde{\theta}))\partial_i\widetilde{u}_j\partial_iU_j{\rm d}\bm{x}.
\eeno
Similarly,
\beno
-\int_\Omega\nabla\cdot(\nu(\hat{\theta})\nabla\hat{\bm{v}}-\nu(\widetilde{\theta})\nabla\widetilde{\bm{v}})\cdot\mathbf{V}{\rm d}\bm{x}
\geq\frac{1}{\sigma}\|\nabla\mathbf{V}\|_{L^2}^2+\int_\Omega(\nu(\hat{\theta})-\nu(\widetilde{\theta}))\partial_i\widetilde{v}_j\partial_iV_j{\rm d}\bm{x},\\
-\int_\Omega\nabla\cdot(\kappa(\hat{\theta})\nabla\hat{\theta}-\kappa(\widetilde{\theta})\nabla\widetilde{\theta})\Theta{\rm d}\bm{x}
\geq\frac{1}{\sigma}\|\nabla\Theta\|_{L^2}^2+\int_\Omega(\kappa(\hat{\theta})-\kappa(\widetilde{\theta}))\partial_i\widetilde{\theta}\partial_i\Theta{\rm d}\bm{x}.
\eeno
Substituting the above estimates into \eqref{uni-est}, we can get
\ben\label{uni-est1}
&&\f{{\rm d}}{{\rm d}t}(\|\mathbf{U}\|_{L^2}^2+\|\mathbf{V}\|_{L^2}^2+\|\Theta\|_{L^2}^2)
+\frac{1}{\sigma}(\|\nabla\mathbf{U}\|_{L^2}^2+\|\nabla\mathbf{V}\|_{L^2}^2+\|\nabla\Theta\|_{L^2}^2)\notag\\
&\leq&Ce^{-\alpha t}(\|\mathbf{U}\|_{L^2}^2+\|\mathbf{V}\|_{L^2}^2+\|\Theta\|_{L^2}^2)
-\int_\Omega(\mu(\hat{\theta})-\mu(\widetilde{\theta}))\partial_i\widetilde{u}_j\partial_iU_j{\rm d}\bm{x}\notag\\
&&-\int_\Omega(\nu(\hat{\theta})-\nu(\widetilde{\theta}))\partial_i\widetilde{v}_j\partial_iV_j{\rm d}\bm{x}
-\int_\Omega(\kappa(\hat{\theta})-\kappa(\widetilde{\theta}))\partial_i\widetilde{\theta}\partial_i\Theta{\rm d}\bm{x}\notag\\
&=&\sum_{i=1}^4I_i.
\een
By direct calculations,
\ben\label{uni-est10}
|I_2+I_3+I_4|
&\leq&\|\nabla\mathbf{U}\|_{L^2}\|\nabla\widetilde{\bm{u}}\|_{L^4}\|\mu(\hat{\theta})-\mu(\widetilde{\theta})\|_{L^4}
+\|\nabla\mathbf{V}\|_{L^2}\|\nabla\widetilde{\bm{v}}\|_{L^4}\|\nu(\hat{\theta})-\nu(\widetilde{\theta})\|_{L^4}\notag\\
&&+\|\nabla\Theta\|_{L^2}\|\nabla\widetilde{\mathbf{\theta}}\|_{L^4}\|\kappa(\hat{\theta})-\kappa(\widetilde{\theta})\|_{L^4},
\een
Also, making use of Lagrange mean value theorem and the fact that $\mu(\theta)$ is a smooth function, we conclude
\ben\label{uni-est11}
&&\|\mu(\hat{\theta})-\mu(\widetilde{\theta})\|_{L^4}\notag\\
&\leq&\|\mu(\hat{\theta})-\mu(\widetilde{\theta})\|_{L^2}^{\frac{1}{2}}
\|\mu'(\hat{\theta})\nabla\hat{\theta}-\mu'(\widetilde{\theta})\nabla\widetilde{\theta}\|_{L^2}^{\frac{1}{2}}
+\|\mu(\hat{\theta})-\mu(\widetilde{\theta})\|_{L^2}\notag\\
&\leq&\|\mu'(\xi)\Theta\|_{L^2}^{\frac{1}{2}}(\|\mu'(\hat{\theta})\nabla\Theta\|_{L^2}^{\frac{1}{2}}
+\|(\mu'(\hat{\theta})-\mu'(\widetilde{\theta}))\nabla\widetilde{\theta}\|_{L^2}^{\frac{1}{2}})
+\|\mu'(\xi)\Theta\|_{L^2}\notag\\
&\leq&C\|\Theta\|_{L^2}^{\frac{1}{2}}(\widetilde {M}^\frac{1}{2}\|\nabla\Theta\|_{L^2}^{\frac{1}{2}}+\|\mu''(\eta)\Theta\nabla\widetilde{\theta}\|_{L^2}^{\frac{1}{2}})
+C\|\Theta\|_{L^2}\notag\\
&\leq&C\|\Theta\|_{L^2}^{\frac{1}{2}}\|\nabla\Theta\|_{L^2}^{\frac{1}{2}}
+C\|\Theta\|_{L^2}\|\nabla\widetilde{\theta}\|_{L^\infty}^{\frac{1}{2}}+C\|\Theta\|_{L^2}\notag\\
&\leq&C\|\Theta\|_{L^2}^{\frac{1}{2}}\|\nabla\Theta\|_{L^2}^{\frac{1}{2}}
+C\|\Theta\|_{L^2}\|\widetilde{\theta}\|_{H^3}^{\frac{1}{2}}+C\|\Theta\|_{L^2},
\een
where $\xi$ and $\eta$ lie between $\hat{\theta}$ and $\widetilde{\theta}$.
Likewise, we hold the same conclusions for $\|\nu(\hat{\theta})-\nu(\widetilde{\theta})\|_{L^4}$ and $\|\kappa(\hat{\theta})-\kappa(\widetilde{\theta})\|_{L^4}$.
Putting \eqref{uni-est11} into \eqref{uni-est10}, one can infer
\ben\label{uni-est12}
&&|I_2+I_3+I_4|\notag\\
&\leq&C(\|\Theta\|_{L^2}^{\frac{1}{2}}\|\nabla\Theta\|_{L^2}^{\frac{1}{2}}
+\|\Theta\|_{L^2}\|\widetilde{\theta}\|_{H^3}^{\frac{1}{2}}+\|\Theta\|_{L^2})
\Big[\|\nabla\mathbf{U}\|_{L^2}(\|\nabla\widetilde{\bm{u}}\|_{L^2}^{\frac{1}{2}}\|\nabla^2\widetilde{\bm{u}}\|_{L^2}^{\frac{1}{2}}+\|\nabla\widetilde{\bm{u}}\|_{L^2})\notag\\
&&+\|\nabla\mathbf{V}\|_{L^2}(\|\nabla\widetilde{\bm{v}}\|_{L^2}^{\frac{1}{2}}\|\nabla^2\widetilde{\bm{v}}\|_{L^2}^{\frac{1}{2}}+\|\nabla\widetilde{\bm{v}}\|_{L^2})
+\|\nabla\Theta\|_{L^2}(\|\nabla\widetilde{\theta}\|_{L^2}^{\frac{1}{2}}\|\nabla^2\widetilde{\theta}\|_{L^2}^{\frac{1}{2}}+\|\nabla\widetilde{\theta}\|_{L^2})\Big]\notag\\
&\leq&Ce^{-\frac{\alpha t}{2}}(\|\nabla\mathbf{U}\|_{L^2}+\|\nabla\mathbf{V}\|_{L^2}+\|\nabla\Theta\|_{L^2})(\|\Theta\|_{L^2}^{\frac{1}{2}}\|\nabla\Theta\|_{L^2}^{\frac{1}{2}}
+\|\Theta\|_{L^2}\|\widetilde{\theta}\|_{H^3}^{\frac{1}{2}}+\|\Theta\|_{L^2})\notag\\
&\leq&\frac{1}{2\sigma}(\|\nabla\mathbf{U}\|_{L^2}^2+\|\nabla\mathbf{V}\|_{L^2}^2+\|\nabla\Theta\|_{L^2}^2)
+Ce^{-\alpha t}(\|\Theta\|_{L^2}^2+\|\Theta\|_{L^2}^2\|\widetilde{\theta}\|_{H^3}).
\een
Inserting \eqref{uni-est12} into \eqref{uni-est1}, we are able to acquire
\ben\label{uni-est2}
&&\frac{{\rm d}}{{\rm d}t}(\|\mathbf{U}\|_{L^2}^2+\|\mathbf{V}\|_{L^2}^2+\|\Theta\|_{L^2}^2)
+\frac{1}{\sigma}(\|\nabla\mathbf{U}\|_{L^2}^2+\|\nabla\mathbf{V}\|_{L^2}^2+\|\nabla\Theta\|_{L^2}^2)\notag\\
&\leq&Ce^{-\alpha t}(\|\mathbf{U}\|_{L^2}^2+\|\mathbf{V}\|_{L^2}^2+\|\Theta\|_{L^2}^2)
+Ce^{-\alpha t}\|\Theta\|_{L^2}^2\|\widetilde{\theta}\|_{H^3}\notag\\
&\leq&C(e^{-\alpha t}+e^{\alpha t}\|\widetilde{\theta}\|_{H^3}^2)(\|\mathbf{U}\|_{L^2}^2+\|\mathbf{V}\|_{L^2}^2+\|\Theta\|_{L^2}^2).
\een
Applying the Gronwall inequality, it yields
\ben\label{uni-est3}
&&\|\mathbf{U}\|_{L^2}^2+\|\mathbf{V}\|_{L^2}^2+\|\Theta\|_{L^2}^2
+\frac{1}{\sigma}\int_0^t(\|\nabla\mathbf{U}\|_{L^2}^2+\|\nabla\mathbf{V}\|_{L^2}^2+\|\nabla\Theta\|_{L^2}^2){\rm d}\tau\notag\\
&\leq&(\|\mathbf{U}_0\|_{L^2}^2+\|\mathbf{V}_0\|_{L^2}^2+\|\Theta_0\|_{L^2}^2)
\exp\Big[{\int_0^t(e^{-\alpha\tau}+e^{\alpha\tau}\|\widetilde{\theta}\|_{H^3}^2){\rm d}\tau}\Big]
=0.
\een
Therefore,
\beno
\mathbf{U}=\mathbf{V}=\Theta\equiv0.
\eeno
This finishes the proof of Theorem \ref{3D}.
\end{proof}

\section*{Acknowledgement}
Liu was partially supported by National Natural Science Foundation of China under grant (No.  11801018, No. 12061003), Beijing Natural Science Foundation under grant (No. 1192001) and Beijing University of Technology under grant (No. 006000514123513).

\end{document}